\documentclass[EJP]{ejpecp}

\usepackage[T1]{fontenc}
\usepackage[utf8]{inputenc}
\usepackage{tikz}
\usepackage[shortlabels]{enumitem}

\SUBMITTED{August 5, 2018}
\ACCEPTED{September 20, 2019}

\VOLUME{24}
\YEAR{2019}
\PAPERNUM{112}
\DOI{10.1214/19-EJP368}

\SHORTTITLE{Infection spread for the frog model on trees}
\TITLE{Infection spread for the frog model on trees}
\AUTHORS{Christopher~Hoffman\footnote{University of Washington.
\EMAIL{hoffman@math.washington.edu}.
Received support from NSF grant DMS-1308645 and from a Simons Fellowship.}
\and
Tobias~Johnson\footnote{College of Staten Island.
\EMAIL{tobias.johnson@csi.cuny.edu}.
Received support from NSF grant DMS-1401479.}
\and
Matthew~Junge\footnote{Bard College.
\EMAIL{mjunge@bard.edu}.
Received support from NSF grant DMS-1855516.}}
\AMSSUBJ{60K35; 60J80; 60J10}
\ABSTRACT{The frog model is an infection process in which dormant particles
begin moving and infecting others once they become infected. We show that on
the rooted $d$-ary tree with particle density $\Omega (d^{2})$, the set of
visited sites contains a linearly expanding ball and the number of visits to
the root grows linearly with high probability.}
\usetikzlibrary{arrows.meta}
\usetikzlibrary{decorations.pathmorphing}
\tikzset{
ragged border/.style={ decoration={random steps, segment length=1mm, amplitude=0.5mm},
decorate,
}
}
\newcommand{\rootvertex}{\varnothing}
\renewcommand{\emptyset}{\varnothing}
\newcommand{\E}{\mathbf{E}}
\renewcommand{\P}{\mathbf{P}}
\newcommand{\TT}{\mathbb{T}}
\newcommand{\ZZ}{\mathbb{Z}}
\newcommand{\Aa}{\mathcal{A}}
\newcommand{\Uu}{\mathcal{U}}
\newcommand{\Vv}{\mathcal{V}}
\newcommand{\RR}{\mathbb{R}}
\newcommand{\NN}{\mathbb{N}}
\newcommand{\Thom}{\TT^{\mathrm{hom}}}
\newcommand{\chibar}{\overline{\chi}}
\DeclareMathOperator{\Poi}{Poi}
\DeclareMathOperator{\Bin}{Bin}
\DeclareMathOperator{\Geo}{Geo}
\DeclarePairedDelimiter{\abs}{\lvert }{\rvert}\DeclarePairedDelimiter{\floor}{\lfloor }{\rfloor}\DeclarePairedDelimiter{\ceil}{\lceil }{\rceil}
\newcommand{\bigmid}{\;\big\vert\;}
\newcommand{\Biggmid}{\ \Bigg\vert\ }
\newcommand{\WA}{\mathop{\mathrm{WA}}}
\newcommand{\sep}{\:\vert\:}
\newcommand{\bigsep}{\:\big\vert\:}
\newcommand{\biggsep}{\ \bigg\vert\ }
\newcommand{\eqd}{\overset{d}=}
\makeatletter
\newcommand{\proc}{\centerdot}
\newcommand{\bigcdot@}[2]{\mathbin{\vcenter{\hbox{\scalebox{#2}{$\m@th#1\bullet$}}}}}
\makeatother

\begin{document}

\section{Introduction}

The \emph{frog model} is a system of interacting walks on a rooted graph
and is one of the most studied examples of the class of $A+B\rightarrow
2B$ models from statistical physics
\cite{TW,KZ,shape,combustion,HJJ1, Fot?}. Particles in this class of
models have one of two states $A$ and $B$. A particle in state $A$
changes to state $B$ on encountering a state~$B$ particle. Once a
particle has state~$B$, it keeps this state for all time. These models
are conservative, meaning that particles are never created nor destroyed.

In the frog model, the particles in state $A$ do not move, while the
particles in state~$B$ perform independent simple random walks in
discrete time. Thus, we refer to the state~$A$ particles as
\emph{asleep} and the state~$B$ particles as \emph{awake} or
\emph{active}. The initial conditions consist of one particle awake at
the root and some numbers of sleeping particles at all other vertices.
The particles are traditionally referred to as frogs, a practice we
continue.

The class of $A+B \rightarrow 2B$ models arose from the study of the
spread of an infection or rumor through a network. Spitzer asked how
fast the infection region (the set of sites visited by a $B$ particle)
grows on $\mathbb{Z}^{d}$. In a series of papers, Kesten and
Sidoravicius gave a partial answer \cite{KS1,KS2, KS3}. Under the
assumption that $A$ and $B$ particles move at the same rate in
continuous time, they proved that the infection region grows linearly
in time and has a limiting shape. That the infection region in the frog
model also grows linearly and has a limiting shape was proven by Alves,
Machado, and Popov in discrete time \cite{shape, random_shape} and
by Ram\'{\i }rez and Sidoravicius in continuous time
\cite{combustion}.

We prove here that if the density of particles is sufficiently large,
the frog model on the infinite $d$-ary tree has a linearly growing
infection region. If the density of particles is small, then linear
growth does not occur, as some sites remain uninfected forever
\cite[Proposition~15]{HJJ2}.

In \cite{HJJ_cover}, we consider the related question of how long
the frog model takes to visit all sites on a finite tree, which we call
the \emph{cover time}. We show the existence of two regimes for the
cover time on the full $d$-ary tree of height~$n$: when the particle
density is $\Omega (d^{2})$, the cover time is $\Theta (n\log n)$; when
the particle density is $O(d)$, the cover time is $\exp (\Theta (
\sqrt{n}))$. See \cite[Theorem~1.1]{HJJ_cover} for a more complete
statement. This paper's results on the infinite tree are an essential
part of our proof of the fast cover time regime on the finite tree. We
apply them to show linear growth of the infected region on finite trees,
far away from leaves. With some more elaborate argument, we then show
that with an extra $O(n\log n)$ time, the leaves are also visited.

\subsection*{Notation}

Given a graph~$G$ and starting vertex, we describe a frog model as a
pair $(\eta ,S)$. For each vertex $v$ other than the starting one,
$\eta (v)$ gives the number of frogs initially sleeping at $v$. The
random variable $S=(S_{\proc } (v,i))_{v\in G,i\geq 1}$ is a collection
of walks satisfying $S_{0}(v,i)=v$. The $i$th particle sleeping at
$v$ on waking follows the path $S_\proc (v,i)$. For the full
construction of the frog model along these lines, see \cite{KZ}.
Typically, $S$ is a collection of simple random walks independent of
each other and of $\eta $, and $(\eta (v))_{v}$ are either i.i.d.\ or
are deterministically equal to some constant, most often $1$. When we
discuss the frog model on a given graph with, say, i.i.d.-$\Poi (
\mu )$ initial conditions, unless we say otherwise we assume that the
paths are simple random walks, and that these independence assumptions
are in effect. For the frog model on a tree, the root is assumed to be
the starting vertex unless stated otherwise. The frog model evolves in
discrete time, though it is easy to show that the results of this paper
hold in continuous time as well (see Remark~\ref{rmk:continuous.time}).
A realization of the frog model is called either \emph{transient} or
\emph{recurrent} depending on whether the starting vertex is visited
infinitely often by frogs. Typically, transience and recurrence are
almost-sure properties; see \cite{KZ}.

We use $\TT _{d}$ to refer to the infinite rooted $d$-ary tree, in which
the root has degree~$d$ and all other vertices have degree $d+1$. We
refer to the vertices at distance~$k$ from the root as level~$k$ of the
tree. We define $\TT _{d}^{n}$ as the full $d$-ary tree of height~$n$,
which is the subset of $\TT _{d}$ made up of levels $0$ to $n$. For any
rooted tree~$T$ and vertex $v\in T$, we denote the subtree of $T$ made
up of $v$ and its descendants by $T(v)$. We use $\varnothing $ to refer
the root of whichever tree we are discussing.

The probability measure $\Geo (p)$ is the distribution on $\{0,1,
\ldots \}$ of the number of failures before the first success in
independent trials that succeed with probability $p$. We also refer to
the geometric distribution on $\{1,2,\ldots \}$ with parameter $p$,
which is the same distribution shifted. In a slight abuse, we will
sometimes use notation like $\Geo (p)$ or $\Bin (n,p)$ to refer not to
a probability distribution but to a random variable with the given
distribution, as in a statement like $\P [\Geo (p) \geq k]=(1-p)^{k}$.

\subsection*{Results}

The first result on the frog model was that with one sleeping frog per
site on $\ZZ ^{d}$, the model is recurrent with probability one for all
$d\geq 1$ \cite{TW}. Next to be pursued were shape theorems for
the process on $\ZZ ^{d}$ demonstrating linear growth in time for the
diameter of the infected region
\cite{shape,random_shape,combustion}. Variations on the frog model with
drift are considered in
\cite{nina_drift1,ghosh_drift,dobler_drift,josh_drift, josh_drift2}. The
recent article \cite{nina_drift2} establishes a phase transition
between transience and recurrence for the one-per-site frog model on
$\mathbb{Z}^{d}$ for $d \geq 2$ as the drift is varied. In
\cite{bmf}, a shape theorem is proven for a Brownian frog model in
Euclidean space. The authors observe a phase transition to superlinear
infection spread at the critical threshold for continuum percolation.

The structure of $\TT _{d}$ induces a natural drift away from the root.
Unlike on the lattice, this drift is counterbalanced by the exponential
increase in the volume of the tree. We have shown that the frog model
exhibits a phase transition between transience and recurrence as $d$ or
the initial configuration is changed. With one sleeping frog per site,
the frog model on a $d$-ary tree is recurrent for $d=2$ and transient
for $d\geq 5$, with the behavior for $d=3,4$ still open
\cite{HJJ1}. Using similar techniques, Rosenberg proves the process is
recurrent on the tree whose levels alternate between two and three
children per vertex \cite{josh_32}. For any $d\geq 2$, the frog
model on a $d$-ary tree with i.i.d.-$\Poi (\mu )$ frogs per site is
transient or recurrent depending on whether $\mu $ is smaller or larger
than a critical value $\mu _{c}(d)>0$ \cite{HJJ2}, and the critical
value satisfies $\mu _{c}(d)=\Theta (d)$ \cite{JJ3_log}.

Our main result is that if $\mu =\Omega (d^{2})$, the infected region
of the frog model on $\TT _{d}$ grows linearly, similar to its behavior
on $\ZZ ^{d}$.

\begin{theorem}
\label{cor:ball}
Consider the frog model on $\TT _{d}$ with i.i.d.-$\Poi (\mu )$ initial
conditions, where $\mu \geq 5d^{2}$. Let $D_{t}$ be the highest level
of the tree at which all vertices have been visited at time~$t$. Then
for some $c,\gamma >0$ depending only on $d$,
\begin{align*}
\P [ D_{t} \leq ct ] \leq e^{-\gamma t}.
\end{align*}
\end{theorem}

We also show that the model is strongly recurrent, in the sense that the
number of visits to the root grows linearly:
%
\begin{theorem}
\label{thm:strong.rec}
Let $V_{t}$ be the number of times the root has been visited up to time
$t$ in the frog model on $\TT _{d}$ with i.i.d.-$\Poi (\mu )$ initial
conditions, where $\mu \geq 5d^{2}$. For some $c$ depending only on
$d$,
%
\begin{align}
\label{eq:strong.rec}
\lim _{t\to \infty } \P [V_{t}\geq ct] = 1.
\end{align}
\end{theorem}

Though our results are stated only for the frog model with Poisson
initial conditions, it is routine to extend them to other initial
conditions using the monotonicity results of \cite{JJ3_order}. See
\cite[Section~4]{JJ3_order} and \cite[Appendix~A]{HJJ_cover} for
examples of such extensions.

While shape theorems for $A+B\to 2B$ models on lattices use the
classical technique of applying the subadditive ergodic theorem, one
needs a different approach on trees. Here, Theorem~$\ref{cor:ball}$
follows without too much difficulty from our statement of strong
recurrence, Theorem~$\ref{thm:strong.rec}$, and its proof is the bulk
of the work. All proofs of recurrence for frog models on trees have used
a bootstrap approach. Essentially, a lower bound on return counts
applied to subtrees is leveraged to obtain an improved lower bound on
return counts in the entire tree. This improved bound for the tree is
then applied again to the subtrees to yield a further improved bound on
the tree, and so on. The argument takes the form of producing an
operator $\Aa $ acting on probability distributions such that if
$\pi $ is the distribution of the return count for each subtree, then
$\Aa \pi $ is the distribution of the return count for the entire tree.
One then argues that $\Aa ^{n}\pi \to \infty $. This approach was used
in \cite{HJJ1,HJJ2,JJ3_log}.

Our proof of strong recurrence in this paper adopts the same approach
in spirit, even using the same operator notation, but in implementation
is completely different from previous arguments. Rather than acting on
return counts, our operator acts on point processes that represent the
time of each return, which are considerably more difficult to work with.
Using stochastic inequalities for point processes, we manage to reduce
the problem to showing that a certain deterministic sequence defined by
a recurrence relation is bounded away from zero. This turns out to be
surprisingly difficult and technical (see Lemma~$\ref{lem:inf.lambda}$).
The proof works only when $\mu =\Omega (d^{2})$, which leaves a gap in
our understanding, as the frog model is transient only when
$\mu =O(d)$. We would be very interested to see a more probabilistic
version of this technical section of the proof. It seems plausible that
this might let it be improved to show strong recurrence for a smaller
value of $\mu $.

\subsection*{Questions}

The most pressing question raised in this paper is the behavior of the
frog model when $d\ll \mu \ll d^{2}$. As we mentioned above, we know
that the frog model on the infinite tree $\TT _{d}$ is transient for
$\mu =O(d)$ \cite[Proposition~15]{HJJ2} and recurrent for
$\mu =\Omega (d)$ \cite[Theorem~1]{JJ3_log}, and in this paper we
show it strongly recurrent when $\mu =\Omega (d^{2})$.
%
\begin{question}
\label{q:weak.recurrence}
Is there a weak recurrence phase for the frog model on $\TT _{d}$, where
the root is occupied for a vanishing fraction of all time and the
infected region grows sublinearly?
\end{question}
 If such a phase exists, it is unclear whether it would occur for an
interval of $\mu $ or just at a single critical value. In
\cite{HJJ_cover}, we ask whether there are more than the two known
regimes for the cover time of finite trees, and whether there are sharp
phase transitions between regimes. We suspect that Question~$
\ref{q:weak.recurrence}$ would be the first step toward resolving those
questions.

\section{Modified frog models}%
\label{sec:modified.frog.models}
As we make our argument, we will usually work with variants of the frog
model where frogs have nonbacktracking paths. We describe these
processes here and relate them back to the usual frog model.

\subsection{The self-similar frog model}%
\label{sec:self.similar}
The \emph{self-similar frog model} on the infinite tree $\TT _{d}$ is
as defined in \cite{JJ3_log}. Put succinctly, it is a modified
version of the frog model with nonbacktracking paths in which all but
the first frog to enter any given subtree are killed. For a more
complete definition, we first define a
\emph{nonbacktracking random walk} (or a \emph{uniform nonbacktracking
random walk} if we wish to distinguish it from other nonbacktracking
walks) from a vertex~$v_{0}$ on any graph with minimum vertex degree
$2$. The walk starts at $v_{0}$. Its first step is to a uniformly random
neighbor of $v_{0}$. On all subsequent steps, it moves to a vertex
chosen uniformly from all its neighbors except the one it just arrived
from.

We now define the self-similar frog model in two steps. First, for each
$v\in \TT _{d}$ and $i\geq 1$, let $S_\proc (v,i)=(S_{j}(v,i))_{j
\geq 0}$ be a random nonbacktracking walk on $\TT _{d}$ starting from
$v$, killed on arrival to $\emptyset $ at times $1$ and beyond. Let
these walks be independent for all $v$ and $i$, and let $S=(S_\proc (v,i))_{v
\in \TT _{d},i\geq 1}$. Let $\eta =(\eta (v),\,v\in \TT _{d})$ be a
given collection of sleeping frog counts. Next, we modify the frog model
$(\eta ,S)$ by killing additional frogs as follows. As the frog model
$(\eta ,S)$ runs, at each step consider all vertices visited for the
first time at that step. Suppose that $v\in \TT _{d}\setminus \{\emptyset
\}$ is one of these vertices, and call its parent $v'$. On this step,
$v$ is necessarily visited by one or more frogs from $v'$. Kill all but
one of these frogs on arrival to $v$, choosing arbitrarily which one
survives. At subsequent steps, kill all frogs moving from $v'$ to
$v$ on arrival to $v$. Also, stop all frogs that move to the root. The
effect is that if a nonroot vertex $v$ is ever visited, then it is
visited only once by a frog originating outside of $\TT _{d}(v)$, and
the frog model on $\TT _{d}(v)$ appears identical in law as the original
one:

\begin{lemma}
[Self-similarity of the self-similar frog model]%
\label{lem:self.similarity}
Let $\emptyset '$ be the child of the root first visited by the
self-similar frog model. Let $v\in \TT _{d}\setminus \{\rootvertex \}$,
and let $v'$ be the parent of $v$. The following two processes are
identical in law:
\begin{enumerate}[(i)]%
\item
the self-similar model viewed on $\{v'\}\cup \TT _{d}(v)$, with frogs
killed on moving from $v$ to $v'$, from the first time of visit to
$v$ onward (conditional on $v$ being visited);
\item
the self-similar model viewed on
$\{\emptyset \}\cup \TT _{d}(\emptyset ')$ from time~$1$ onward.
\end{enumerate}
\end{lemma}

\subsection{The nonbacktracking frog model}%
\label{sec:nonbacktracking}
First, define a \emph{root-biased nonbacktracking random walk from
$v_{0}$ on $\TT _{d}$} as a walk distributed as follows. We set
$X_{0}=v_{0}$, and then we choose $X_{1}$ uniformly from the neighbors
of $X_{0}$. Conditionally on $X_{0},\ldots ,X_{i}$, we choose
$X_{i+1}$ as follows: If $X_{i}=\varnothing $, choose $X_{i+1}$ to be
$X_{i-1}$ with probability $1/d^{2}$ and to be each of the other
children of the root with probability $(d+1)/d^{2}$. Otherwise, choose
$X_{i+1}$ uniformly from the neighbors of $X_{i}$ other than
$X_{i-1}$. It turns out that this describes the behavior of the walk
that results from deleting all excursions of a simple random walk away
from its eventual path (see Appendix~\ref{sec:decompose}). The odd
behavior at the root results from the asymmetry of the tree there. We
then define the \emph{nonbacktracking frog model on $\TT _{d}$} as the
frog model whose paths are independent root-biased nonbacktracking
random walks on $\TT _{d}^{n}$.

Recalling that $\TT _{d}^{n}$ consists of levels $0,\ldots ,n$ of
$\TT _{d}$ define a \emph{root-biased nonbacktracking random walk from
$v_{0}$ on $\TT _{d}^{n}$} just as above, except that when $X_{i}$ is
a leaf of $\TT _{d}^{n}$, we define $X_{i+1}$ to be the parent of
$X_{i}$. Define the
\emph{nonbacktracking frog model on $\TT ^{n}_{d}$} to have these walks
as its paths.

The following result, which we prove in Appendix~\ref{sec:decompose},
demonstrates that the time change for the underlying random walks has
little effect. This allows us to work with nonbacktracking frog models
when we prove Theorem~$\ref{cor:ball}$.
%
\begin{proposition}
\label{cor:time.change}
Let $(\eta ,S)$ and $(\eta ,S')$ be respectively the usual and the
nonbacktracking frog models, both on $\TT _{d}$ or both on $\TT ^{n}
_{d}$, with arbitrary initial configuration $\eta $. There exists a
coupling of the frog models $(\eta ,S)$ and $(\eta ,S')$ such that the
following holds: For any $b>\log d$, there exists $C=C(b)$ such that all
vertices visited in $(\eta ,S')$ by time~$t$ are visited in
$(\eta ,S)$ by time~$Ct$ with probability $1-e^{-bt}$.
\end{proposition}

\begin{remark}
\label{rmk:continuous.time}
To extend the results of this paper to continuous time, one could prove
a version of this proposition in which a continuous-time frog model is
coupled with a discrete-time nonbacktracking frog model.
\end{remark}

We mention that frog models on the finite tree $\TT ^{n}_{d}$ do not
come up in this paper. We include such models in Proposition~$
\ref{cor:time.change}$ because the proof is nearly identical for them,
and we need the result for our paper \cite{HJJ_cover} on cover
times for the frog model.

\section{Strong recurrence for the self-similar model}%
\label{sec:strong.recurrence}
In this section, we show that in the self-similar frog model on
$\TT _{d}$ with sufficiently large initial density of frogs, the visits
to the root are bounded from below by a Poisson process. Our main
results, Theorems $\ref{cor:ball}$ and~$\ref{thm:strong.rec}$, will
follow in the next section as easy corollaries.

To state the result, define the \emph{return process} of a given frog
model to be a point process consisting of a point at $t$ for each frog
occupying the root at time $t\geq 1$. Note that the return process will
be supported on the positive even integers, as frogs never visit the
root at odd times by periodicity. For point processes $\xi _{1}$ and
$\xi _{2}$, we say that $\xi _{1}$ is stochastically dominated by
$\xi _{2}$ (denoted $\xi _{1}\preceq \xi _{2}$) if there exists a coupling
of $\xi _{1}$ and $\xi _{2}$ such that every point of $\xi _{1}$ is
contained in $\xi _{2}$. See Section~\ref{sec:dominance.point.processes}
for more background material on stochastic dominance between point
processes.

\begin{theorem}
\label{thm:strong}
Consider the self-similar frog model on $\TT _{d}$ with i.i.d.-$
\Poi (\mu )$ initial conditions. For any $d\geq 2$, $\alpha >0$, and
$\mu \geq 3d(d+1)+\alpha (d+1)$, the return process stochastically
dominates a Poisson point process with intensity measure $\sum _{k=1}
^{\infty }\alpha \delta _{2k}$.
\end{theorem}

Ram\'{\i }rez and Sidoravicius proved that the field of site
occupation counts for the one-per-site frog model on $\ZZ ^{d}$ in
continuous time converges to independent Poisson random variables with
mean one \cite[Theorem~1.2]{combustion}. Compared to this result,
ours provides useful information across times, but it only gives lower
bounds. Though our result is stated for occupation times at the root
only, by Lemma~$\ref{lem:self.similarity}$ it can be applied to any site
starting at the time of its first visit.

The proof of Theorem~$\ref{thm:strong}$ is in three steps, Lemmas
$\ref{lem:theta_n}$, $\ref{lem:chi_n}$, and~$\ref{lem:inf.lambda}$. Let
$\theta $ be the return process of the frog model in Theorem~$
\ref{thm:strong}$. First, we show that $\theta \succeq \Aa ^{n}0$ for
all $n$, where $\Aa $ is a certain operator acting on the laws of point
processes and $0$ is the empty point process. Second, we prove by
induction that $\Aa ^{n}0$ dominates a Poisson point process with
intensity measure $\sum _{k=1}^{n}\lambda _{k}\delta _{2k}$ for an explicit
sequence $(\lambda _{k})_{k\geq 1}$ depending on $\mu $. The final step
is to show that $\lambda _{k}$ remains bounded away from zero as
$k\to \infty $. Though $\lambda _{k}$ can be explicitly computed given
$\mu $, its formula is a complicated expression involving $\lambda
_{1},\ldots ,\lambda _{k-1}$, and it takes some effort to show that
$\lambda _{k}$ does not decay to zero as $k$ tends to infinity. As we
mentioned in the introduction, this argument is a sort of bootstrap.
Every time we iterate the operator $\Aa $, the self-similarity of the
frog model together with a lower bound on the return process yields an
improved lower bound on the return process.

\subsection{Stochastic dominance for point processes}%
\label{sec:dominance.point.processes}
Let $X$ and $Y$ be random variables. We say that $X$ is stochastically
dominated by $Y$, notated $X\preceq Y$, if $\P [X\geq x]\leq \P [Y
\geq x]$ for all $x\in \RR $. This is equivalent to the existence of a
coupling of $X$ and $Y$ where $X\leq Y$ a.s.

If $X\preceq Y$, then $\P [X=0]\geq \P [Y=0]$ by the definition of
stochastic dominance. If $X$ is Poisson and $Y$ is a mixture of
Poissons, then the converse is true as well:
%
\begin{lemma}
[{\cite[Theorem~3.1(b)]{MSH}}]%
\label{lem:MSH}
Let $X\sim \Poi (\lambda )$, and let $Y\sim \Poi (U)$ for some
nonnegative random variable~$U$. Then the following are equivalent:
\begin{enumerate}[(i)]%
\item
$X\preceq Y$,
\label{i:st}
\item
$\P [X=0]\geq \P [Y=0]$, and
\label{i:zero.prob}
\item
$\lambda \leq -\log \E e^{-U}$.
\label{i.log}
\end{enumerate}
\end{lemma}
 See also \cite[Section~2]{JJ3_log} for a quick presentation of the
proof from \cite{MSH}.

A \emph{Cox process} is a mixture of Poisson point processes. The next
result is a sufficient condition for stochastic dominance of a Poisson
process by a Cox process in the same spirit as the previous lemma.
Notationally, we view a point process $\xi $ as a nonnegative
integer--valued random measure on $\RR $. For $U\subseteq \RR $, we
denote the restriction of the point process to $U$ by $\xi |_{U}$. We
write the number of atoms of $\xi $ found in $U$ as $\xi (U)$,
abbreviating $\xi (\{x_{1},\ldots ,x_{n}\})$ to $\xi \{x_{1},\ldots ,x
_{n}\}$.

\begin{lemma}
\label{prop:Cox.lower}
Let $\xi $ be a Cox process supported on a countable or finite set
$\{x_{1},x_{2},\ldots \}$. Let $\Uu _{n}=\{x_{1},\ldots ,x_{n}\}$.
Suppose that
\begin{align*}
\P \bigl [ \xi \{x_{1}\}=0 \bigr ]
&\leq p_{1},
\\
\intertext{and for all $n\geq 2$,}
\P \Bigl [ \xi \{x_{n}\}=0 \;\Big \vert \; \xi |_{\Uu _{n-1}}\Bigr ]
&
\leq p_{n}\text{ a.s.}
\end{align*}
Then $\xi $ stochastically dominates a Poisson point process with
intensity measure
\begin{align*}
\sum _{n=1}^{\infty } (-\log p_{n})\delta _{x_{n}}.
\end{align*}
\end{lemma}
\begin{proof}
This follows from \cite[Theorem~6.B.3]{SS} combined with Lemma~$
\ref{lem:MSH}$. This amounts to coupling $\xi \{x_{1}\}$ with a Poisson
random variable, then conditioning on $\xi \{x_{1}\}$ and coupling
$\xi \{x_{2}\}$ with an independent Poisson random variable, then
conditioning on $\xi \{x_{1}\}$ and $\xi \{x_{2}\}$ and coupling
$\xi \{x_{3}\}$ with an independent Poisson, and so on.
\end{proof}

\subsection{The operator $\Aa $ and its connection to the frog model}

Our first task is to define an operator~$\Aa =\Aa _{d,\mu }$ acting on
distributions of point processes. For a point process $\xi $ with distribution~$
\nu $, we will abuse notation slightly and write $\Aa \xi $ rather than
$\Aa \nu $.

Let us first explain the idea behind $\Aa $. The initial frog in the
self-similar frog model moves from the root $\varnothing $ down the
tree, first to a child $\varnothing '$, and then to one of its children,
$v_{1}$. Let $v_{2},\ldots ,v_{d}$ be the other children of
$\varnothing '$. The basic idea is to imagine the frog model as
occurring only on these vertices. When a frog moves to $x_{i}$, we close
our eyes to the subtree $\TT _{d}(x_{i})$, imagining it as a black box
that will occasionally emit frogs going back towards the root. We then
define $\Aa \xi $ as the return process of this frog model if we assume
that each black box emits frogs at times given by an independent copy
of $\xi $, shifted by the time of activation.

To formally define $\Aa \xi $, consider a modified frog process taking
place on a star graph with center $\rho '$ and leaf vertices
$\rho ,u_{1},\ldots ,u_{d}$ (see Figure~\ref{fig:star}). One should
think of $\rho $ and $\rho '$ as paralleling $\varnothing $ and
$\varnothing '$, and of $u_{1},\ldots ,u_{d}$ as as paralleling
$v_{1},\ldots ,v_{d}$. We define $\Aa \xi $ as the distribution of the
point process of times of visits to $\rho $ in the process defined as
follows.
\begin{enumerate}[(i)]%
\item
At time~$0$, there is a single particle awake at $\rho $ and
$\Poi (\mu )$ sleeping particles at $\rho '$. There are also particles
asleep on each of $u_{1},\ldots ,u_{d}$, as will be described in
\ref{i:u_i}.
\item
The initially active particle moves from $\rho $ to $\rho '$ to
$u_{1}$ and then halts.
\item
The particles at $\rho '$ are woken at time~$1$ by the initial particle.
In the next step, each particle moves independently to a neighbor chosen
uniformly at random. All particles halt after this step.
\item
When site $u_{i}$ is visited, the particles there undergo a delayed
activation as follows. For each $i$, let $\xi ^{(i)}$ be an independent
copy of $\xi $. For each atom in $\xi ^{(i)}$, there is a sleeping
particle on $u_{i}$. If the atom is at $k$, then this sleeping particle
is awoken $k-2$ steps after $u_{i}$ is first visited, and it makes its
first move (necessarily to $\rho '$) one time step after this. In its
next step, it chooses uniformly at random from the neighbors of
$\rho '$ except for $u_{i}$ (i.e., it takes a random nonbacktracking
step). It then halts.
\label{i:u_i}
\end{enumerate}

\begin{figure}
\begin{center}
\begin{tikzpicture}[xscale=3,yscale=2,vert/.style={circle,fill,inner sep=0,
minimum size=0.125cm,draw},
active/.style={rectangle, fill, inner sep=0, minimum size=0.25cm,draw},
sleeping/.style={rectangle, inner sep=0, minimum size=0.25cm,draw}]

\path (0,0) node[vert,label=below left:$\rho '$] (root') {}
(-1,.2) node[vert,label=below:$\rho $] (root) {}
(1,-0.8) node[vert,label=below:$u_{1}$] (v1) {}
(1,-0.15) node[vert,label=below:$u_{2}$] (v2) {}
(1,.8) node[vert,label=below:$u_{d}$] (vd) {}
(1,.34) node {$\vdots $};

\draw [thick] (root) -- (root') -- (v1);
\path[every node/.style={font=\sffamily \small }]
(root') edge [-{To[]},bend right=35,blue] node[fill=white] {$\textcolor{black}{\Aa \xi }$} (root)
(v1) edge [-{To[]},bend left=35,blue] node[fill=white] {$\textcolor{black}{\xi ^{(1)}}$} (root')
(v2) edge [-{To[]},bend right=30,blue] node[fill=white] {$\textcolor{black}{\xi ^{(2)}}$} (root')
(vd) edge [-{To[]},bend right=35,blue] node[fill=white] {$\textcolor{black}{\xi ^{(d)}}$} (root')
;

\draw [dashed] (root')--(v2)
(root')--(vd);
\end{tikzpicture}

\caption{The point process $\mathcal{A} \xi $ records the visit times to $\rho $
in a point process that behaves somewhat like the frog model. The difference is that
when $u_{i}$ is first visited, particles are released not immediately but at the times
in $\xi ^{(i)}$, which is an independent copy of the point process $\xi $.
One should think of this system as a frog
model where we ignore the activity past level~$2$ of the tree, paying attention
only to when particles emerge back to level~$1$.}
\label{fig:star}
\end{center}
\end{figure}
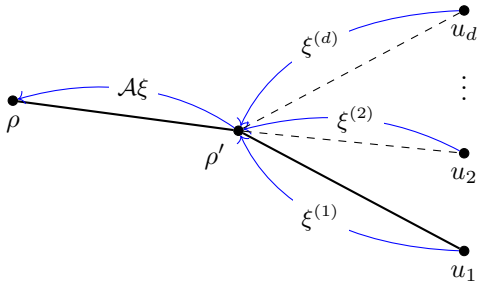

The next lemma gives the connection between this operator $\Aa $ and the
frog model. The proof is in the same in spirit as
\cite[Lemma~3.5]{JJ3_log}.
%
\begin{lemma}
\label{lem:theta.fixed}
Let $\theta $ be the return process considered in Theorem~$
\ref{thm:strong}$. Then $\theta \eqd \Aa \theta $.
\end{lemma}
\begin{proof}
We couple the self-similar frog model with the particle system defining
$\Aa \theta $ as follows. Couple the number of frogs at $\rho '$ in the
particle system with the number of frogs initially at $\emptyset '$, and
couple the first step of each particle with its corresponding frog. Also
suppose without loss of generality that the initial frog moves to
$v_{1}$ after $\emptyset '$. At time~$2$, then, a frog moves to
$v_{1}$ and a particle moves to $u_{1}$. By Lemma~$
\ref{lem:self.similarity}$, if $v_{i}$ is visited, then the self-similar
frog model from this time on restricted to $\emptyset '\cup \TT _{d}(v
_{i})$ is distributed as the original frog model on $\emptyset \cup \TT
_{d}(\emptyset ')$. Thus, the visit times to $\emptyset '$ by frogs
originating within $\TT _{d}(v_{1})$ are distributed as $\theta $,
shifted forward in time by one step (the first possible visit to
$\emptyset '$ is at time~$3$, while the first possible atom of
$\theta _{n}$ is at $2$). This matches the definition of $\Aa \theta $,
and so we can couple these return times to the times of visits of
particles to $\rho '$ from $u_{1}$ in the particle system. We also
couple their next steps. We then continue in this way, coupling return
times of frogs to independent shifted copies of $\theta $ whenever one
of the sites $v_{2},\ldots ,v_{d}$ is visited. Ultimately, the return
times to $\varnothing $ in the frog model are identical to the visits
to $\rho $ in the particle system, proving that $\theta \eqd \Aa
\theta $.
\end{proof}

\begin{lemma}
\label{lem:monotone}
If $\xi \preceq \xi '$, then $\Aa \xi \preceq \Aa \xi '$.
\end{lemma}
\begin{proof}
Couple the number and paths of the particles initially at $\rho '$ in
the particle systems defining $\Aa \xi $ and $\Aa \xi '$. Couple
$\xi ^{(i)}$ and $\xi '^{(i)}$ for $i=1,\ldots ,d$ so that every point
in $\xi ^{(i)}$ is contained in $\xi '^{(i)}$, and couple the paths of
the corresponding particles to be identical. Now, every visit to
$\rho $ at time~$k$ in the particle system defining $\Aa \xi $ also
occurs in the particle system defining $\Aa \xi '$, demonstrating that
$\Aa \xi \preceq \Aa \xi '$.
\end{proof}

\begin{lemma}
\label{lem:theta_n}
Let $\theta $ be the return process in Theorem~$\ref{thm:strong}$. For
all $n\geq 0$, we have $\theta \succeq \Aa ^{n}0$.
\end{lemma}
\begin{proof}
We prove this by induction on $n$. Trivially, the lemma holds when
$n=0$. Now, assuming that $\theta \succeq \Aa ^{n}0$, we apply Lemmas
$\ref{lem:theta.fixed}$ and~$\ref{lem:monotone}$ to deduce that
\begin{align*}
\theta \eqd \Aa \theta \succeq \Aa (\Aa ^{n}0)
&=\Aa ^{n+1}0.\qedhere
\end{align*}
\end{proof}

If $\xi $ is a Poisson point process, we can take advantage of Poisson
thinning to give a tractable lower bound on $\Aa \xi $. Given a point
process $\xi $, let $\tau \xi $ denote the $1/d$--thinning of
$\xi $, a point process that includes each atom of $\xi $ independently
with probability~$1/d$. Let $\sigma _{t}\xi $ denote the result of
shifting each atom in $\xi $ by $t$ (for example, $\sigma _{t}(\delta
_{2}+\delta _{5})=\delta _{2+t}+\delta _{5+t}$). Given a sequence of
nonnegative numbers $(\lambda _{k})_{k\geq 1}$, define $S$ to be the
number of failures until a success in mixed Bernoulli trials where the
first trial has success probability $1- e^{-\mu /(d+1)}$, and the
$(k+1)$th has success probability $1 - e^{-\lambda _{k}/d}$ for
$k\geq 1$. Alternatively, we can write $S$ as the random variable on
$\{0,1,\ldots \}\cup \{\infty \}$ with distribution given by
%
\begin{align}
\label{eq:S.def}
\begin{split}
\P [S=0]
&= 1 - e^{-\mu /(d+1)},
\\
\P [S=k\mid S\geq k]
&= 1-e^{-\lambda _{k}/d},
\qquad
k\geq 1.
\end{split}
\end{align}
We will typically have $\lambda _{k}=0$ for $k > n$. Under this
condition, $S$ is supported on $\{0,\ldots ,n\}\cup \{\infty \}$.

\begin{proposition}
\label{lem:Aa.PPP}
Suppose that $\xi $ is a Poisson point process with intensity measure
$\sum _{i=1}^{\infty }\lambda _{i}\delta _{2i}$. Let $S^{(2)},S^{(3)},
\ldots $ be independent copies of $S$ defined above, let $\xi ^{(1)},
\xi ^{(2)},\ldots $ be independent copies of $\xi $, and let
$Z\sim \Poi \bigl (\mu /(d+1)\bigr )$ be independent of all of these. Then
%
\begin{align}
\label{eq:Aa.PPP}
\Aa \xi \succeq Z\delta _{2} + \sigma _{2}\tau \xi ^{(1)} + \sum _{i=2}
^{d}\sigma _{2+2S^{(i)}}\tau \xi ^{(i)},
\end{align}
where $\tau $ and $\sigma _{t}$ are the thinning and shift operators
defined above.
\end{proposition}
\begin{proof}
Consider the particle system defining $\Aa \xi $. Suppose that we
disallow particles starting at vertices $u_{2},\ldots ,u_{d}$ from
waking up any of sites $u_{2},\ldots ,u_{d}$. Since there are
stochastically fewer particles, the process of return times to
$\rho $ in this system is stochastically dominated by $\Aa \xi $. We now
show that these return times are distributed as the right-hand side of
\eqref{eq:Aa.PPP}.

The first term in \eqref{eq:Aa.PPP} counts the particles initially
sleeping at $\rho '$ that move to $\rho $ on step~2. The second term
counts the particles that emerge from $u_{1}$ and move to $\rho '$ and
then $\rho $. The other terms count the particles emerging from
$u_{2},\ldots ,u_{d}$ and moving to $\rho '$ and then $\rho $. We now
explain why these terms are as in \eqref{eq:Aa.PPP} and why they are
independent.

There are initially $\Poi (\mu )$ particles on vertex~$\rho '$ in the
particle system. These particles are woken on step~$1$ and then move to
random neighbors in step~$2$. This is the source of the $Z\delta _{2}$
term.

By definition of the particle system, particles from $u_{i}$ arrive at
$\rho '$ at times $k-1$ after $u_{i}$ is activated, for each atom
$k$ in $\xi ^{(i)}$. Each of these particles moves next to $\rho $ with
probability~$1/d$. Hence, $\rho $ is visited at times $\tau \xi ^{(i)}$
after activation of $u_{i}$. In particular, this explains the form of
the second term on the right-hand side of \eqref{eq:Aa.PPP}, as
$u_{1}$ is activated at time~$2$.

Last, we consider the time of activation of each of $u_{2},\ldots ,u
_{d}$. These sites can be activated either by particles initially at
$\rho '$ or by particles emerging from $u_{1}$. For $i=2,\ldots ,d$, let
$Z_{i}$ be the number of particles initially at $\rho '$ that move to
$u_{i}$. Similarly, let $\tau _{i}\xi ^{(1)}$ denote the point process
that keeps the points of $\xi ^{(1)}$ corresponding to particles that
move from $\rho '$ to $u_{i}$. By Poisson thinning, $Z,Z_{2},\ldots ,Z
_{d}$ are independent, and $\tau \xi ^{(1)},\tau _{2}\xi ^{(1)},\ldots ,
\tau _{d}\xi ^{(1)}$ are independent, and both are independent of each
other as well. This explains the independence of the terms on the
right-hand side of \eqref{eq:Aa.PPP}. The first visit to $u_{i}$ is the
first point of
\begin{align*}
Z_{i}\delta _{2} + \sigma _{2}\tau _{i}\xi ^{(1)},
\end{align*}
and it is easily seen that this is distributed as $2+2S$. As
$\rho $ is visited by particles from $u+i$ at times $\tau \xi ^{(i)}$
after activation of $u_{i}$, this explains the terms in the summand on
the right-hand side of \eqref{eq:Aa.PPP}.
\end{proof}

\subsection{Iterating $\Aa $ to prove strong recurrence}%
\label{sec:iterating}
By Lemma~$\ref{lem:theta_n}$, a lower bound on $\Aa ^{n}0$ is a lower
bound on the return process of the self-similar frog model. Thus, all
it takes to prove Theorem~$\ref{thm:strong}$ are suitable lower bounds
on $\Aa ^{n}0$. We provide them inductively, by applying Proposition~$
\ref{lem:Aa.PPP}$ and Lemma~$\ref{prop:Cox.lower}$ to extend a lower
bound on $\Aa ^{n}0$ to $\Aa ^{n+1}0$. This argument deals only with
point process; one can completely forget about the frog model for the
remainder of the section.

For a fixed choice of $\mu $ and $d$, we will define a sequence
$(\lambda _{n})_{n\geq 1}$. We then define a point process $\chi _{n}$
with intensity measure $\sum _{k=1}^{n}\lambda _{k}\delta _{2k}$, which
will serve as the lower bound for $\Aa ^{n}0$ (see Lemma~$
\ref{lem:chi_n}$). Our definition of $(\lambda _{n})$ is in terms of
another sequence $(P_{n})_{n\geq 1}$ defined by a recurrence relation.
As we will prove immediately after giving the definitions, the two
sequences $(P_{n})$ and $(\lambda _{n})$ are related by
%
\begin{align}
\label{eq:lambda.link}
P_{n} = \exp \biggl (-\frac{1}{d}\sum _{k=1}^{n}\lambda _{k}\biggr ).
\end{align}
Thus, for any $n\leq m$, we think of $\lambda _{n}$ as the intensity of
the Poisson point process $\chi _{m}$ at the point $2n$, and we think of
$P_{n}$ as the probability that a $\frac{1}{d}$-thinned version of
$\chi _{m}$ contains no points at $\{2, 4, \ldots , 2n\}$. The reason
that the definitions are given in terms of $(P_{n})$ is to set us up to
apply Lemma~$\ref{prop:Cox.lower}$.

\begin{definition}
[$\lambda _{n}$, $P_{n}$, and $\chi _{n}$]%
\label{def:sequences}
Let $a=e^{-\mu /(d+1)}$. Define
%
\begin{align}
P_{1}
&= a^{1/d},
\label{eq:P_1}%
\\
P_{2}
&= P_{1}\Bigl [ (1-a)P_{1} + a \Bigr ],
\label{eq:P_2}
\end{align}
and for all $n\geq 2$,
%
\begin{align}
\label{eq:P_n}
\begin{split}
P_{n+1}
&= P_{1}\Bigl [ (1-a)P_{n} + a\Bigl ( (1-P_{1})P_{n-1} + (P_{1}-P
_{2})P_{n-2}
+\cdots
\\
&
\qquad
\qquad
\qquad
\qquad
\qquad
\qquad
\qquad
+ (P_{n-2}-P_{n-1})P_{1} + P_{n-1}\Bigr ) \Bigr ].
\end{split}
\end{align}
Next, we define $(\lambda _{n})_{n\geq 1}$ by $\lambda _{1}=\mu /(d+1)$
and
%
\begin{align}
\label{eq:def.lambda}
\lambda _{n}
&= -d\log \biggl (\frac{P_{n}}{P_{n-1}}\biggr ),\quad n
\geq 2.
\end{align}
Finally, define $\chi _{n}$ to be a Poisson point process with intensity
measure $\sum _{k=1}^{n}\lambda _{k}\delta _{2k}$.
\end{definition}

Though we have explained how $(P_{n})$ and $(\lambda _{n})$ should be
interpreted, nothing about them is apparent from their definitions. We
start analyzing their relationship by proving \eqref{eq:lambda.link}:
it holds for $n=1$ by definition and can be extended inductively by
applying \eqref{eq:def.lambda}. Next, we consider some basic properties
of the two sequences:
%
\begin{proposition}
\label{prop:sequence_easy}
The sequence $(P_{n})_{n\geq 1}$ is nonnegative and decreasing, and the
sequence $(\lambda _{n})_{n\geq 1}$ is nonnegative.
\end{proposition}
\begin{proof}
Since $(1-a)P_{1} + a$ is a convex combination of $P_{1} < 1$ and
$1$, it is less than $1$. This shows that $P_{2}\leq P_{1}$. Using
\eqref{eq:P_2} and \eqref{eq:P_n},
\begin{align*}
P_{2}-P_{3}
&= P_{1}\biggl [ (1-a)(P_{1}-P_{2}) + a\biggl (1 - \Bigl ((1-P
_{1})P_{1} + P_{1}\Bigr )\biggr )\biggr ].
\end{align*}
\vspace*{-14pt}
\eject\noindent
Since $P_{2}\leq P_{1}$ and $(1-P_{1})P_{1}+P_{1}\leq 1$ by the same
convex combination argument, the above expression is nonnegative,
showing that $P_{3}\leq P_{2}$. Now, assume that $P_{n}\leq \cdots
\leq P_{1}$, and we will show that $P_{n+1}\leq P_{n}$. By
\eqref{eq:P_n},
\begin{align*}
P_{n}-P_{n+1}
&= P_{1}\Biggl [ (1-a)(P_{n-1}-P_{n})
\\
&
\qquad
\quad \quad
+a\biggl ( (1-P_{1})(P_{n-2}-P_{n-1}) + (P_{1}-P_{2})(P_{n-3}-P_{n-2})
+\cdots
\\
&
\qquad
\qquad
\qquad
\qquad
+ (P_{n-3}-P_{n-2})(P_{1}-P_{2}) +(P_{n-2}-P_{n-1})(1-P_{1})\biggr )
\Biggr ].
\end{align*}
The inductive hypothesis shows that all factors are nonnegative, proving
$P_{n+1}\leq P_{n}$. Hence, $(P_{n})_{n\geq 1}$ is decreasing. Thus, all
terms in \eqref{eq:P_n} are nonnegative, showing that $P_{n}\geq 0$ for
all $n$. Combined with \eqref{eq:def.lambda}, this proves that
$\lambda _{n}\geq 0$.
\end{proof}

\begin{proposition}
\label{prop:sequence_hard}
The sequence $(\lambda _{n})_{n\geq 1}$ is decreasing.
\end{proposition}
The proof of this is a lengthy and unilluminating sequence of algebraic
manipulations that we have left to Appendix~\ref{sec:painful.appendix}.
We are left with the feeling that there must be a probabilistic
explanation, but we could not come up with it.

In the next lemma, we see that $\chi _{n}$ gives us a lower bound on
$\Aa ^{n}0$ and hence on the return process of the self-similar frog
model. We will see that the definition of $(P_{n})$ is tailored to this
lemma.

\begin{lemma}
\label{lem:chi_n}
For all $n\geq 0$, it holds that $\Aa ^{n}0\succeq \chi _{n}$.
\end{lemma}
\begin{proof}
We will show that $\Aa \chi _{n}\succeq \chi _{n+1}$ for all $n\geq 0$.
The proposition then follows by repeatedly applying Lemma~$
\ref{lem:monotone}$. By Proposition~$\ref{lem:Aa.PPP}$, it suffices to
show that
%
\begin{align}
\label{eq:to.prove}
Z\delta _{2} + \sigma _{2}\tau \chi _{n}^{(1)} + \sum _{i=2}^{d}
\sigma _{2+2S^{(i)}}\tau \chi _{n}^{(i)}\succeq \chi _{n+1}.
\end{align}
Here $Z\sim \Poi \bigl (\mu /(d+1)\bigr )$, $\chi _{n}^{(i)}$ is
distributed as $\chi _{n}$, $S^{(i)}$ is distributed according to
\eqref{eq:S.def}, and all are mutually independent. Thus, all terms on
the left-hand side of \eqref{eq:to.prove} are independent. The terms
$Z\delta _{2}$ and $\sigma _{2}\tau \chi _{n}^{(1)}$ are Poisson point
processes, while the other terms are Cox processes. The work will be in
proving the following about the Cox processes:
%
\begin{claim}
The point process $\sigma _{2+2S^{(i)}}\tau \chi ^{(i)}_{n}$
stochastically dominates a Poisson point process with intensity measure
$\sum _{k=2}^{n+1}\tfrac{\lambda _{k}}{d}\delta _{2k}$.
\end{claim}
\begin{proof}
For the sake of brevity, let $\chibar =\sigma _{2+2S^{(i)}}\tau \chi
^{(i)}_{n}$. As mentioned, $\chibar $ is a Cox process, since it is
Poisson conditional on $S$. Our goal is to apply Lemma~$
\ref{prop:Cox.lower}$ to it. With this in mind, for $2\leq k\leq n$ we
compute
\begin{align*}
\P \Bigl [ \chibar \{4,6,\ldots ,2k+2\}=0\Bigr ]
&= \sum _{s=0}^{k-1} \P
\bigl [S^{(i)}=s\bigr ]\exp \biggl (-\frac{1}{d} \sum _{i=1}^{k-s}\lambda
_{i}\biggr )
+ \P \bigl [S^{(i)}\geq k\bigr ].
\end{align*}
Expanding this using the distribution of $S^{(i)}$ from
\eqref{eq:S.def}, applying \eqref{eq:lambda.link}, and then applying
\eqref{eq:P_n},
\begin{align*}
\P \Bigl [ \chibar \{4,6,\ldots ,
&2k+2\}=0\Bigr ]
\\
&= (1-e^{-\mu /(d+1)})\exp \biggl (-\frac{1}{d} \sum _{i=1}^{k}\lambda
_{i}\biggr )
\\
&
\qquad
\qquad
+ \sum _{s=1}^{k-1}\exp \biggl (-\frac{\mu }{d+1}-\frac{1}{d}\sum _{i=1}
^{s-1}\lambda _{i}\biggr )
(1-e^{-\lambda _{s}/d})
\exp \biggl (-
\frac{1}{d} \sum _{i=1}^{k-s}\lambda _{i}\biggr )
\\
&
\qquad
\qquad
+\exp \biggl (-\frac{\mu }{d+1}-\frac{1}{d}\sum _{i=1}^{k-1}\lambda _{i}
\biggr )
\\
&= (1-a)P_{k} + a\Biggl [(1-P_{1})P_{k-1} + \sum _{s=2}^{k-1}(P_{s-1}-P
_{s})P_{k-s} + P_{k-1}\Biggr ]
\\
&=\frac{P_{k+1}}{P_{1}}.
\end{align*}
The same calculation shows that the final conclusion holds in the
$k=1$ case, and it holds trivially in the $k=0$ case. Hence, for
$1\leq k\leq n$,
%
\begin{align}
\P \Bigl [ \chibar \{2k+2\}=0 \;\Big \vert \;
\chibar \{4,\ldots ,2k\}=0
\Bigr ]
&= \frac{P_{k+1}/P_{1}}{P_{k}/P_{1}}=e^{-\lambda _{k+1}/d}.
\label{eq:Cox.worstcase}
\end{align}
We claim that this bound on the probability of $\chibar \{2k+2\}=0$
holds conditional on \emph{any} point configuration for $\chibar $ on
$\{4,\ldots ,2k\}$. That is, for all $1\leq k\leq n$,
%
\begin{align}
\label{eq:Cox.criterion}
\P \Bigl [ \chibar \{2k+2\}=0 \;\Big \vert \;
\chibar \big |_{\{4,
\ldots ,2k\}} \Bigr ]
&\leq e^{-\lambda _{k+1}/d} \text{ a.s.}
\end{align}
To prove this, suppose that $\chibar \big |_{\{4,\ldots ,2k\}}$ contains
any points. This guarantees that $S^{(i)}\leq k-2$. For $s\leq k-2$ (or
indeed even for $s\leq k-1$),
\begin{align*}
\P \Bigl [\chibar \{2k+2\}=0\;\Big \vert \;S^{(i)}=s\Bigr ]
= e^{-
\lambda _{k-s}/d} \leq e^{-\lambda _{k+1}/d},
\end{align*}
since $(\lambda _{i})_{i\geq 1}$ is decreasing by Proposition~$
\ref{prop:sequence_hard}$. This shows that conditional on $\chibar
\big |_{\{4,\ldots ,2k\}}$ being equal to any nonempty collection of
points, the probability that $\chibar \{2k+2\}=0$ is bounded by
$e^{-\lambda _{k+1}/d}$. Together with \eqref{eq:Cox.worstcase}, this
confirms \eqref{eq:Cox.criterion}. By Lemma~$\ref{prop:Cox.lower}$, this
proves the claim.
\end{proof}

To finish the proof of Lemma~$\ref{lem:chi_n}$, let $\xi ^{(1)},\xi
^{(2)},\ldots $ denote Poisson point processes with intensity measures
$\sum _{k=2}^{n+1}\tfrac{\lambda _{k}}{d}\delta _{2k}$, independent of each
other and all else. By the claim,
%
\begin{align}
\label{eq:clm.cons}
Z\delta _{2} + \sigma _{2}\tau \chi _{n}^{(1)} + \sum _{i=2}^{d}
\sigma _{2+2S^{(i)}}\tau \chi _{n}^{(i)}
\succeq Z\delta _{2} + \sigma
_{2}\tau \chi _{n}^{(1)} + \sum _{i=2}^{d}\xi ^{(i)}.
\end{align}
The right-hand side is Poisson with intensity measure
\begin{align*}
\lambda _{1}\delta _{2} + \sum _{k=2}^{n+1}\frac{\lambda _{k-1}+(d-1)
\lambda _{k}}{d}\delta _{2k}.
\end{align*}
Since $\lambda _{k-1}\geq \lambda _{k}$ by Proposition~$
\ref{prop:sequence_hard}$, the right-hand side of \eqref{eq:clm.cons}
stochastically dominates $\chi _{n+1}$. This confirms
\eqref{eq:to.prove} and completes the proof.
\end{proof}

With Lemmas $\ref{lem:theta_n}$ and~$\ref{lem:chi_n}$ established, all
that remains is to show that $\chi _{n}$ dominates a Poisson point
process with intensity measure $\sum _{k=1}^{n}\alpha \delta _{2k}$ for
some $\alpha >0$. This amounts to showing that $\lambda _{k}\geq
\alpha >0$ for all $k\geq 1$. We give a technical lemma and then the
proof.

\begin{lemma}
\label{lem:sum}
For $\beta >1$,
\begin{align*}
\sum _{k=1}^{n}\Bigl ( k (n+1-k)\Bigr )^{-\beta } \leq 2^{\beta +1}
\zeta (\beta )n^{-\beta },
\end{align*}
where $\zeta (\beta )=\sum _{k=1}^{\infty }k^{-\beta }$.
\end{lemma}
\begin{proof}
By symmetry,
\begin{align*}
\sum _{k=1}^{n}\Bigl ( k (n+1-k)\Bigr )^{-\beta }
&\leq 2\sum _{k=1}^{
\lceil n/2\rceil } \Bigl ( k (n+1-k)\Bigr )^{-\beta }
\\
&\leq 2\sum _{k=1}^{\lceil n/2\rceil }(kn/2)^{-\beta } \leq 2^{\beta
+1} \zeta (\beta )n^{-\beta }.\qedhere
\end{align*}
\end{proof}

\begin{lemma}
\label{lem:inf.lambda}
For any $\gamma >0$, if $\mu \geq (\gamma +3)d(d+1)$, then $\inf _{k}
\lambda _{k}\geq d\gamma $.
\end{lemma}
\begin{proof}
We will show that if $\mu \geq (\gamma +3)d(d+1)$, then
%
\begin{align}
\label{eq:exponential}
P_{k}\leq e^{-\gamma (k-1)-2\log k},
\qquad
k\geq 1.
\end{align}
This holds for $k=1$, since by definition, $P_{1}=e^{-\mu /d(d+1)}
\leq 1$. For $k=2$, using Proposition~$\ref{prop:sequence_easy}$ we have
\begin{align*}
P_{2}\leq P_{1} = e^{-\mu /d(d+1)}\leq e^{-(\gamma +3)}\leq e^{-
\gamma -2\log 2}.
\end{align*}

We now proceed inductively, assuming that \eqref{eq:exponential} holds
for $1\leq k\leq n$ for some $n\geq 2$, and showing that it holds for
$k=n+1$ as well. By definition,
\begin{align*}
P_{n+1}
&= P_{1}\Biggl [ (1-a)P_{n} + a\Biggl ( (1-P_{1})P_{n-1} +
\sum _{i=2}^{n-1} (P_{i-1} - P_{i})P_{n-i}
+ P_{n-1}\Biggr )\Biggr ]
\\
&\leq P_{1}\Biggl [ (1-a)P_{n} + a\Biggl (2P_{n-1} + \sum _{i=2}^{n-1}P
_{i-1}P_{n-i} \Biggr )\Biggr ]
\\
&\leq P_{1}\Biggl [ (1-a)e^{-\gamma (n-1)-2\log n}
\\
&
\qquad
\qquad
\qquad
+ a\Biggl ( 2e^{-\gamma (n-2)-2\log (n-1)}
+ \sum _{i=2}^{n-1}e^{-
\gamma (n-3)-2\log (i-1) -2\log (n-i)}\Biggr )\Biggr ]
\\
&=e^{-\gamma n - 2\log (n+1)}P_{1}\Biggl [(1-a)e^{\gamma }\biggl (
\frac{n+1}{n}\biggr )^{2}
\\
&
\qquad
\qquad
\qquad
\qquad
\qquad
+a\Biggl (2e^{2\gamma }\biggl (\frac{n+1}{n-1}\biggr )^{2}
+e^{3\gamma }(n+1)^{2}
\sum _{i=2}^{n-1}\Bigl ((i-1)(n-i)\Bigr )^{-2}
\Biggr )\Biggr ].
\end{align*}
For all $n\geq 2$,
\begin{align*}
(1-a)e^{\gamma }\biggl (\frac{n+1}{n}\biggr )^{2}
&\leq 3e^{\gamma },
\\
\intertext{and}
2e^{2\gamma }\biggl (\frac{n+1}{n-1}\biggr )^{2}
&\leq 18e^{2\gamma }.
\end{align*}
By Lemma~$\ref{lem:sum}$, for all $n\geq 3$,
\begin{align*}
e^{3\gamma }(n+1)^{2}\sum _{i=2}^{n-1}\Bigl ((i-1)(n-i)\Bigr )^{-2}
&
\leq 8\zeta (2)e^{3\gamma } \biggl (\frac{n+1}{n-2}\biggr )^{2}<211 e
^{3\gamma },
\end{align*}
and in the $n=2$ case the left hand side of this inequality is zero.
Thus
%
\begin{align}
\label{eq:Pn+1}
P_{n+1}
&\leq e^{-n - 2\log (n+1)}P_{1}\bigl (3e^{\gamma }+a(18e^{2
\gamma } + 211 e^{3\gamma })\bigr ).
\end{align}
Recall that $P_{1}=e^{-\mu /d(d+1)}$ and $a=e^{-\mu /(d+1)}$. Since
$\mu >(\gamma +3)d(d+1)$, we have $P_{1}< e^{-(\gamma +3)}$ and
$a< e^{-(\gamma +3)d}$. Hence, for all $d\geq 2$,
\begin{align*}
P_{1}\bigl (3e^{\gamma }+a(18e^{2\gamma } + 211 e^{3\gamma })\bigr )
&
\leq e^{-(\gamma +3)}(3e^{\gamma }+ e^{-(\gamma +3)d}(18 e^{2\gamma }
+ 211 e^{3\gamma })\bigr )
\\
&\leq 3e^{-3} + 18e^{-\gamma -9} + 211 e^{-9}\leq 1.
\end{align*}
This bound with \eqref{eq:Pn+1} completes the induction, establishing
\eqref{eq:exponential}.

By Proposition~$\ref{prop:sequence_easy}$, the sequence $\lambda _{k}$
is nonnegative, and by Proposition~$\ref{prop:sequence_hard}$, it is
decreasing. Thus it has a limit. Suppose the limit is strictly smaller
than $d\gamma $. Recalling that $\lambda _{k}=-d\log (P_{k}/P_{k-1})$,
we then have $P_{k}/P_{k-1}\geq e^{-\gamma +\epsilon }$ for all
sufficiently large $k$ and some $\epsilon >0$. Thus,
\begin{align*}
\liminf _{k\to \infty } \frac{1}{k}\log P_{k} > -\gamma .
\end{align*}
But this contradicts \eqref{eq:exponential}. Therefore $
\lim _{k\to \infty }\lambda _{k}\geq d\gamma $.
\end{proof}

\begin{proof}[Proof of Theorem~$\ref{thm:strong}$]
Let $\xi _{n}$ be a Poisson point process with intensity measure\break
$\sum _{k=1}^{n}\alpha \delta _{2k}$, for $n\in \NN \cup \{\infty \}$. By
Lemmas $\ref{lem:theta_n}$ and~$\ref{lem:chi_n}$, we have $\theta
\succeq \chi _{n}$ for all $n\in \NN $. Applying Lemma~$
\ref{lem:inf.lambda}$ with $\gamma =\alpha /d$, we have $\chi _{n}
\succeq \xi _{n}$. Thus $\theta \succeq \xi _{n}$ for all $n\in \NN $. By
\cite[Theorem~6.B.30]{SS}, which asserts the equivalence of our
definition of stochastic dominance to stochastic dominance of all
finite-dimensional distributions, we have $\theta \succeq \xi _{\infty
}$.
\end{proof}

\section{Proofs of the main results}%
\label{sec:applications}
Theorem~$\ref{thm:strong}$ and Proposition~$\ref{cor:time.change}$
combine for a quick proof of Theorem~$\ref{thm:strong.rec}$:
\begin{proof}[Proof of Theorem~$\ref{thm:strong.rec}$]
Let $U_{t}$ be the number of visits to the root by time~$t$ in the
self-similar frog model on $\TT _{d}$ with i.i.d.-$\Poi (\mu )$ initial
conditions. It follows from $\mu \geq 5d^{2}$ and $d\geq 2$ that
$\mu \geq 3d(d+1) + \frac{2}{3}(d+1)$. By Theorem~$\ref{thm:strong}$,
\begin{align*}
U_{t}\succeq \Poi \bigl (\tfrac{2}{3}\floor{t/2}\bigr ).
\end{align*}
Applying Proposition~$\ref{prop:Poi.bound}$ or even just Chebyshev's
inequality,
%
\begin{align}
\label{eq:U_t.bound}
\lim _{t\to \infty } \P [U_{t}\leq t/4]=0.
\end{align}

Now, let $V_{t}$ be the number of visits to the root by time~$t$ in the
usual frog model. Since $U_{t}$ is stochastically smaller than the
corresponding statistic in the nonbacktracking frog model on
$\TT _{d}$ (recall its definition from
Section~\ref{sec:nonbacktracking}), Proposition~$
\ref{cor:time.change}$ shows there exists $C=C(d)$ such that
$\P [V_{Ct}< U_{t}]\to 0$ as $t\to \infty $. With
\eqref{eq:U_t.bound}, this shows that
\begin{align*}
\lim _{t\to \infty }\P [V_{Ct}\leq t/4]=0.
\end{align*}
Choosing $c$ appropriately in \eqref{eq:strong.rec}, this completes the
proof.
\end{proof}

Next, we build toward Theorem~$\ref{cor:ball}$. The following lemma is
also helpful for our cover time results in \cite{HJJ_cover}.

\begin{lemma}
\label{lem:speed}
Consider the self-similar frog model on $\TT _{d}$ with i.i.d.-$
\Poi (\mu )$ frogs per site, where $\mu =(3+\beta )d(d+1)$ for arbitrary
$\beta >0$. Consider a ray $\varnothing , v_{0}, v_{1},v_{2},\ldots $
and condition the initial frog to take its first step to $v_{0}=
\varnothing '$. Let $\tau _{i}$ be the number of steps after
$v_{i-1}$ is first visited that $v_{i}$ is first visited. Then
$(\tau _{i})_{i\geq 1}$ are i.i.d.\ and satisfy
%
\begin{align}
\label{eq:speed}
\P [\tau _{i} > 2t-1]
&\leq e^{-\beta t}
\end{align}
for $t\in \{1,2,\ldots \}$.
\end{lemma}
\begin{proof}
It is a consequence of Lemma~$\ref{lem:self.similarity}$ that
$(\tau _{i})_{i\geq 1}$ are i.i.d. For the tail bound, we first observe
that $\tau _{1}=1$ if a frog at $\varnothing '$ moves immediately to
$v_{1}$. The initial frog does this with probability $1/d$, and one of
the frogs starting at $\varnothing '$ does so with probability
$1-e^{-\mu /(d+1)}$. Hence,
%
\begin{align}
\label{eq:tau1}
\P [\tau _{1} > 1]
&= \frac{d}{d+1}e^{-\mu /(d+1)}\leq e^{-\beta d}
\leq e^{-\beta },
\end{align}
which proves \eqref{eq:speed} when $t=1$. If neither of these events
occurs and $\tau _{1}>1$, then some sibling $v_{1}'$ of $v_{1}$ is
visited one step after $\varnothing '$ is visited. By Lemma~$
\ref{lem:self.similarity}$ and Theorem~$\ref{thm:strong}$, the number
of visits from $v_{1}'$ to $\varnothing '$ in the first $2t$ steps after
activation of $\varnothing '$ is stochastically larger than
$\Poi (\beta d t)$. Each of these frogs moves next to $v_{1}$ with
probability $1/d$. By Poisson thinning, the number of visits to
$v_{1}$ in the first $2t+1$ steps after activation of $\varnothing '$
is stochastically larger than $\Poi (\beta t)$. Thus,
%
\begin{align}
\label{eq:missed.bound}
\P [ \tau _{1}>2t+1\mid \tau _{1}>1]
&\leq e^{-\beta t}.
\end{align}
Equations~\eqref{eq:tau1} and \eqref{eq:missed.bound} show that
$\P [\tau _{1}>2t+1]\leq e^{-\beta (t+1)}$, proving \eqref{eq:speed} for
$t\geq 2$.
\end{proof}

This lemma shows that the time to wake up a given vertex at level~$k$
is something like the sum of $k$ geometric random variables. Thus, by
a union bound, all vertices at level~$k$ are likely to be visited in
time $O(k)$. We now make this formal to prove our main result.
\begin{proof}[Proof of Theorem~$\ref{cor:ball}$]
For applying Lemma~$\ref{lem:speed}$, we observe that our condition
$\mu \geq 5d^{2}$ implies that $\mu \geq (3+\beta )d(d+1)$ where
$\beta =1/3$. Thus, we take $\beta =1/3$ for this entire proof.

Start with a self-similar frog model on $\TT _{d}$, and let
$\emptyset '$ be the child of the root visited on the first step. We now
make a change to this process to allow frogs outside of
$\TT _{d}(\emptyset ')$ to be visited. Rather than killing all frogs at
the root at time~$1$ on, allow them to continue moving as root-biased
nonbacktracking walks (see Section~\ref{sec:nonbacktracking}),
reflecting with probability $1/d^{2}$ and moving to the other children
of the root with probability $(d+1)/d^{2}$ each. Since this process
continues to follow the rule that only a single frog is allowed to enter
any subtree, a frog at the root is still stopped if its next move is to
a previously visited child of the root. The resulting frog model is then
a stopped version of the nonbacktracking frog model on $\TT _{d}$, which
we can relate back to the usual frog model via Proposition~$
\ref{cor:time.change}$.

Consider an arbitrary path $\emptyset ,v_{0},\ldots ,v_{k-1}$ from the
root outward in $\TT _{d}$. Define $\tau _{0}$ as the time that
$v_{0}$ is first visited, and then for $i\geq 1$ define $\tau _{i}$ as
the number of steps it takes to visit $v_{i}$ after $v_{i-1}$ is first
visited. The time to visit $v_{k-1}$ is then $\tau _{0}+\cdots +\tau
_{k-1}$. The time $\tau _{0}$ does not fit the criteria of Lemma~$
\ref{lem:speed}$ exactly, because a frog that moves from a sibling of
$v_{0}$ back to $\emptyset $ moves next to $v_{0}$ with probability
$(d+1)/d^{2}$ rather than $1/d$ as in Lemma~$\ref{lem:speed}$. But this
only makes it more likely to visit $v_{0}$, and so it still satisfies
the bound that $\tau _{i}$ does in \eqref{eq:speed}, by the same
argument. Regardless of how long it takes to visit $v_{0}$, once it is
visited, the process restricted to $\TT _{d}(v_{0})$ is identical in
law. Thus, $\tau _{0}$ is independent of $\tau _{1},\ldots ,\tau _{k-1}$.

From the time that $v_{0}$ is visited, the model restricted to
$\{\emptyset \}\cup \TT _{d}(v_{0})$ with frogs stopped at $\emptyset $
is a self-similar frog model. With the annoyance of dealing with the
asymmetry of the root behind us, we apply Lemma~$\ref{lem:speed}$ to
conclude that the random variables $(\tau _{i})$ are i.i.d.\ and have
exponential tails whose decay is a fixed constant (since $\beta =1/3$).
By Proposition~$\ref{prop:exp.sum.bound}$, vertex~$v_{k-1}$ is unvisited
at time $Ck$ with probability at most $d^{-3k}$, for $C$ depending only
on $d$.

We now take a union bound over all $d^{k}$ choices of $v_{k-1}$ at
level~$k$, showing that there is an unvisited level~$k$ vertex at time
$Ck$ with probability at most $d^{-2k}$. By Proposition~$
\ref{cor:time.change}$, it holds with probability $1-d^{-2k}$ that all
vertices visited in our process by time $Ck$ are also visited in the
usual frog model on $\TT _{d}$ within time $C'k$, for a large enough
choice of $C'$ depending on $d$. Thus,
\begin{align*}
\P [ D_{C'k}\leq k] \leq 2d^{-2k}\leq d^{-k}.
\end{align*}
Setting $k=\ceil{t/C'}$ completes the proof.
\end{proof}


\appendix

\section{Excursion decomposition of random walks on trees}%
\label{sec:decompose}
The goal of this section is to prove Proposition~$
\ref{cor:time.change}$ by breaking down a random walk on a tree into a
loop-erased portion plus excursions. We carry this out first for a
random walk on $\Thom _{d}$, which denotes the $(d+1)$-homogeneous tree,
the infinite tree in which all vertices have degree~$d+1$. The
decomposition for random walks on the less symmetric tree $\TT _{d}$
will follow as a corollary. Though we do not need it for this paper, we
also work out the decomposition for walks on the finite tree
$\TT _{d}^{n}$, which we use in \cite{HJJ_cover}.

Given neighboring vertices $u,v\in \Thom _{d}$, we define a \emph{$\Thom
_{d}$-excursion} from $u$ with first step to $v$ as a random walk on
$\Thom _{d}$ defined as follows. The walker begins at $u$ and takes its
first step to $v$. On subsequent steps before returning to $u$, move
towards $u$ with probability $d/(d+1)$, and move to the $d$ other
possible neighbors each with probability $1/d(d+1)$. The
$\Thom _{d}$-excursion ends when $u$ is reached, which will occur almost
surely.

Our next proposition decomposes a simple random walk on $\Thom _{d}$
into a nonbacktracking random walk spine (recall the definition from
Section~\ref{sec:self.similar}) with independent $\Thom _{d}$-excursions
off of it. We believe this decomposition must be known in some form, but
we have not managed to find a reference to it. Given two paths
$x_{0},\ldots ,x_{a}$ and $y_{0},\ldots ,y_{b}$ such that $x_{a}=y
_{0}$, we define the concatenation of the first and second paths as
$x_{0},\ldots ,x_{a},y_{1},\ldots ,y_{b}$. Note that the concatenation
of the two walks $x_{0},\ldots ,x_{a}$ and $x_{a}$ leaves the first walk
unaffected.

\begin{proposition}
\label{prop:hom.decomposition}
Let $(X_{i})_{i\geq 0}$ be a uniform nonbacktracking random walk on
$\Thom _{d}$ from $y_{0}$. Define another walk $(Y_{i})_{i\geq 0}$ by
\begin{align*}
(Y_{i})_{i\geq 0}
&=
(X_{0},E^{0}_{1},\ldots ,E^{0}_{\ell _{0}}, X_{1},E
^{1}_{1},\ldots ,E^{1}_{\ell _{1}},X_{2},E^{2}_{1},\ldots ,E^{2}_{\ell
_{2}},\ldots ),
\end{align*}
where conditionally on $(X_{i})_{i\geq 0}$, the paths $(E_{i}^{j})_{i=1}
^{\ell _{j}}$ are independent for different $j$ and distributed as
follows:
\begin{description}%
\item[First step] Let $G\sim \Geo \bigl ((d-1)/d\bigr )$ be independent
of all else. Define $(X_{0},E^{0}_{1},\ldots ,E^{0}_{\ell _{0}})$ to be
the concatenation of $G$ independent $\Thom _{d}$-excursions from
$X_{0}$ with first step chosen uniformly from the neighbors of
$y_{0}$.
\item[Subsequent steps] Let $G\sim \Geo \bigl (d/(d+1)\bigr )$ be
independent of all else. For all $j\geq 1$, define $(X_{j},E^{j}_{1},
\ldots ,E^{j}_{\ell _{j}})$ to be the concatenation of $G$ independent
$\Thom _{d}$-excursions from $X_{j}$ with first step chosen uniformly
from the neighbors of $X_{j}$ other than $X_{j-1}$.
\end{description}
Note that $\ell _{j}=0$ if the random variable $G$ used to construct it
is zero.

Then $(Y_{i})_{i\geq 0}$ is simple random walk on $\Thom _{d}$ from
$y_{0}$.
\end{proposition}

\begin{proof}
Let $Z_{0},\ldots ,Z_{K}$ be the geodesic from $Y_{0}$ to $Y_{n}$.
Define $J\in \{0,\ldots ,K\}$ by
\begin{align*}
J = \min \bigl \{j\in \{0,\ldots ,K\} \colon Z_{j}\in \{Y_{n},Y_{n+1},
\ldots \}\bigr \}.
\end{align*}
In other words, $Z_{J}$ is the farthest toward $y_{0}$ the walk will
reach from time~$n$ onward. We claim that for any nearest-neighbor walk
$y_{0},\ldots ,y_{n}$,
%
\begin{align}
\label{eq:comb.exercise}
\P \bigl [
\text{$J=j$ and $(Y_{0},\ldots ,Y_{n})=(y_{0},\ldots ,y_{n})$}\bigr ]
&=
\begin{cases}
d^{-k}(d+1)^{-n}
& \text{if $j=0$,}
\\
(d-1)d^{-k+j-1}(d+1)^{-n}
&\text{if $1\leq j\leq k$.}
\end{cases}
\end{align}
To prove this, let $z_{0},\ldots ,z_{k}$ be the geodesic from
$y_{0}$ to $y_{n}$. Given $J=j$ and $(Y_{0},\ldots ,Y_{n})=(y_{0},
\ldots ,y_{n})$, we can classify the steps of the walk into three
categories:
\begin{description}%
\item[Permanently forward steps]\ \\
The step from $y_{i}$ to $y_{i+1}$ is \emph{permanently forward} if it
moves away from $y_{0}$, the distance from $y_{i}$ to $y_{0}$ is greater
than $j$, and the vertex $y_{i}$ does not appear again in $y_{i+1},
\ldots ,y_{n}$. This implies that $(y_{i},y_{i+1})=(z_{i'},z_{i'+1})$
for some $i'>j$ and that the walk will never revisit $y_{i}$.

A permanently forward step from $y_{i}=y_{0}$ to $y_{i+1}$ occurs with
probability $(d-1)/d(d+1)$, since the walk has probability $(d-1)/d$ of
taking a permanently forward step and probability $1/(d+1)$ of taking
this step to $y_{i+1}$. A permamently forward step from $y_{i}\neq y
_{0}$ to $y_{i+1}$ has probability $1/(d+1)$, since the walk has
probability $d/(d+1)$ of taking a permanently forward step and
probability $1/d$ of taking this step to $y_{i+1}$.
\item[Excursion-forward steps]\ \\
The step from $y_{i}$ to $y_{i+1}$ is \emph{excursion-forward} if it
moves away from $y_{0}$ and is not permanently forward (i.e., either
$y_{i}$ appears in $y_{i+1},\ldots ,y_{n}$ or the distance from
$y_{i}$ to $y_{0}$ is $j$ or less). This implies that the step is part
of an excursion.

Each excursion-forward step occurs with probability $1/d(d+1)$. To prove
this, suppose the step is from $y_{i}$ to $y_{i+1}$. If the walk has
taken an excursion-forward step prior to step~$i$ and has not yet
finished the $\Thom _{d}$-excursion, then this holds by the dynamics of
a $\Thom _{d}$-excursion. If $y_{i}=y_{0}$, then the walk has
probability $1/d$ of starting a $\Thom _{d}$-excursion and probability
$1/(d+1)$ of taking its next step to $y_{i+1}$. If $y_{i}=z_{i'}$ for
some $i'>0$ and it is not in the midst of a $\Thom _{d}$-excursion, then
the walk has probability $1/(d+1)$ of starting a $\Thom _{d}$-excursion
and probability $1/d$ of taking its next step to $y_{i+1}$.
\item[Backward steps]\ \\
The step from $y_{i}$ to $y_{i+1}$ is \emph{backward} if it is toward
$y_{0}$. Every backward step occurs in the midst of a
$\Thom _{d}$-excursion and has probability $d/(d+1)$.
\end{description}

We now use this decomposition of the walk to prove
\eqref{eq:comb.exercise}. Let $f$ and $b$ be the number of forward and
backward steps, respectively, in $y_{0},\ldots ,y_{n}$. Since
$f-b=k$ and $f+b=n$, we have $f=(n+k)/2$ and $b=(n-k)/2$. If $j=0$, then
all forward steps are excursion-forward, and
\begin{align*}
\P \bigl [
\text{$J=j$ and $(Y_{0},\ldots ,Y_{n})=(y_{0},\ldots ,y_{n})$}\bigr ]
&=
\bigl (d(d+1)\bigr )^{-f}\biggl (\frac{d}{d+1}\biggr )^{b}
\\
&=d^{-\frac{n+k}{2} + \frac{n-k}{2}}(d+1)^{-\frac{n+k}{2}-
\frac{n-k}{2}}
\\
&= d^{-k}(d+1)^{-n}.
\end{align*}
If $j>0$, then there are $j$ permanently forward steps, one of which is
from $y_{0}$. The rest of the forward steps are excursion-forward. Thus,
\begin{align*}
\P \bigl [
\text{$J\!=\!j$ and $(Y_{0},\ldots ,Y_{n})=(y_{0},\ldots ,y_{n})$}\bigr ]
&=
\frac{d-1}{d(d+1)}(d+1)^{-(j-1)}\bigl (d(d+1)\bigr )^{-(f-j)}\biggl (\!\frac{d}{d+1}\biggr )^{b}
\\
&= (d-1)d^{-1-\frac{n+k}{2}+j+\frac{n-k}{2}}(d+1)^{-\frac{n+k}{2}-
\frac{n-k}{2}}
\\
&= (d-1)d^{-k+j-1}(d+1)^{-n}.
\end{align*}
This establishes \eqref{eq:comb.exercise}. Therefore,
\begin{align*}
\P \bigl [ (Y_{0},\ldots ,Y_{n})=(y_{0},\ldots ,y_{n}) \bigr ]
&=
d^{-k}(d+1)^{-n} + \sum _{j=1}^{k}(d-1)d^{-k+j-1}(d+1)^{-n}
= (d-1)^{-n}.
\end{align*}
Thus, the law of $(Y_{0},\ldots ,Y_{n})$ is uniform over all
nearest-neighbor paths from $y_{0}$.
\end{proof}

Next, we obtain similar decompositions of random walks on the infinite
tree $\TT _{d}$ and on the finite tree $\TT _{d}^{n}$. These
decompositions are slightly more complicated as a result of the
asymmetry of the trees. First, we embed $\TT _{d}$ into $\Thom _{d}$ as
follows. Designate one vertex in $\Thom _{d}$ as the root $\emptyset $.
Specify one of its neighbors as its parent and the other $d$ as its
children. Then consider $\TT _{d}$ to consist of $\emptyset $ and all
of its descendants. We also embed $\TT _{d}^{n}$ into $\TT _{d}$ in the
obvious way.

Now, consider a walk $(X_{i})_{i\geq 0}$ in $\Thom _{d}$. We define a
new walk as follows. Delete all portions of $(X_{i})$ that lie outside
of $\TT _{d}$. This might result in the walk remaining at the root for
consecutive steps, in which case we delete all but one of these steps
to keep it a nearest-neighbor walk. We call this the \emph{restriction
of $(X_{i})_{i\geq 0}$ to $\TT _{d}$}. Completely analogously, we define
the restriction of $(X_{i})_{i\geq 0}$ to $\TT ^{n}_{d}$. Observe that
the restriction of $(X_{i})$ to $\TT _{d}$ will be a finite path if
$(X_{i})$ escapes to infinity in the direction of the parent of the root
in $\Thom _{d}$. The restriction of $(X_{i})$ to $\TT ^{n}_{d}$ is
always finite, so long as $(X_{i})$ escapes to infinity.

Given neighboring vertices $u,v\in \TT _{d}$, we define a \emph{$\TT
_{d}$-excursion from $u$ with first step to $v$} as the restriction to
$\TT _{d}$ of a $\Thom _{d}$-excursion from $u$ with first step to
$v$. Similarly, for any neighboring vertices $u,v\in \TT ^{n}_{d}$, a
\emph{$\TT ^{n}_{d}$-excursion} from $u$ with first step to $v$ is the
restriction of a $\TT _{d}$-excursion from $u$ with first step to
$v$.

When we talk of a random walk randomly stopped at a given vertex with
probability~$p$, we mean that on each visit to the vertex, the walk is
stopped with probability~$p$ independent of all else.

\begin{lemma}
\label{lem:Td.restriction}
Let $(X_{i})_{i\geq 0}$ and $(Y_{i})_{i\geq 0}$ be the walks on
$\Thom _{d}$ defined in Proposition~$\ref{prop:hom.decomposition}$. Let
$(\overline{X}_{i})_{i=0}^{S}$ and $(\overline{Y}_{i})_{i=0}^{T}$ be the
restrictions of these to $\TT _{d}$. Then the following holds:
\begin{enumerate}[label=(\roman*)]%
\item
The walk $(\overline{X}_{i})_{i=0}^{S}$ is a uniform random
nonbacktracking walk stopped at the root randomly with probability
$1/(d+1)$ at time~$0$ and with probability $1/d$ on subsequent visits
to the root.
\label{i:Td.restriction.X}
\item
The walk $(\overline{Y}_{i})_{i=0}^{T}$ is a simple random walk on
$\TT _{d}$ stopped at the root randomly with probability $(d-1)/(d
^{2}+d-1)$.
\label{i:Td.restriction.Y}
\item
When $S<\infty $,
\label{i:Td.restriction.excursions}
\begin{align*}
(\overline{Y}_{i})_{i=0}^{T}
&= (\overline{X}_{0},E_{1}^{0},\ldots ,E
_{\ell _{0}}^{0},\overline{X}_{1},
E_{1}^{1},\ldots ,E_{\ell _{1}}^{1},
\ldots , \overline{X}_{S},E_{1}^{S},\ldots ,E_{\ell _{S}}^{S}),
\end{align*}
and when $S=\infty $,
\begin{align*}
(\overline{Y}_{i})_{i=0}^{T}
&= (\overline{X}_{0},E_{1}^{0},\ldots ,E
_{\ell _{0}}^{0},\overline{X}_{1},
E_{1}^{1},\ldots ,E_{\ell _{1}}^{1},
\overline{X}_{2},\ldots ).
\end{align*}
Conditional on $(\overline{X}_{i})_{i=0}^{S}$ and on $S$, the walks
$(E^{j}_{i})_{i=1}^{\ell _{j}}$ are independent and distributed as
follows for $0\leq j \leq S$:
\begin{enumerate}[label=(\alph*)]%
\item
If $\overline{X}_{0}=\emptyset $, then $(\overline{X}_{0},E^{0}_{1},
\ldots ,E^{0}_{\ell _{0}})$ is the concatenation of $\Geo \bigl ((d^{2}-1)/(d
^{2}+d-1)\bigr )$ many independent $\TT _{d}$-excursions from
$\emptyset $ with first step uniformly chosen from the children of
$\emptyset $.%
\label{i:Td.restriction.a}
\item
If $\overline{X}_{0}\neq \emptyset $, then $(\overline{X}_{0},E^{0}
_{1},\ldots ,E^{0}_{\ell _{0}})$ is the concatenation of $\Geo \bigl ((d-1)/d
\bigr )$ many independent $\TT _{d}$-excursions from $\overline{X}_{0}$
with first step chosen uniformly from the neighbors of $\overline{X}
_{0}$.
\label{i:Td.restriction.b}
\item
If $\overline{X}_{j}=\emptyset $ for $j\geq 1$, then $(\overline{X}
_{j},E^{j}_{1},\ldots ,E^{j}_{\ell _{j}})$ is the concatenation of
$\Geo \bigl (d^{2}/(d^{2}+d-1)\bigr )$ many independent
$\TT _{d}$-excursions from $\emptyset $ with first step chosen uniformly
from the children of $\emptyset $ other than $\overline{X}_{j-1}$.%
\label{i:Td.restriction.c}
\item
If $\overline{X}_{j}\neq \emptyset $ for $j\geq 1$, then $(
\overline{X}_{j},E^{j}_{1},\ldots ,E^{j}_{\ell _{j}})$ is the
concatenation of $\Geo \bigl (d/(d+1)\bigr )$ many independent
$\TT _{d}$-excursions from $\overline{X}_{j}$ with first step chosen
uniformly from the neighbors of $\overline{X}_{j}$ other than
$\overline{X}_{j-1}$.%
\label{i:Td.restriction.d}
\end{enumerate}
\end{enumerate}
\end{lemma}
\begin{proof}
By Proposition~$\ref{prop:hom.decomposition}$, $(X_{i})_{i=0}^{\infty
}$ is a uniform nonbacktracking random walk on $\Thom _{d}$. The walk
$(\overline{X}_{i})_{i=0}^{S}$ follows the same path, except that if
$(X_{i})$ moves outside of $\Thom _{d}$, then $(\overline{X}_{i})_{i=0}
^{S}$ is stopped on the last step before it does so. This has a
$1/(d+1)$ chance of occurring on the first step if the walk starts at
the root, and it has a $1/d$ chance of occurring on subsequent visits
to the root. This proves \ref{i:Td.restriction.X}.

Similarly, $(Y_{i})_{i=0}^{\infty }$ is a simple random walk on
$\Thom _{d}$. Thus, $(\overline{Y}_{i})_{i=0}^{T}$ is a simple random
walk on $\TT _{d}$, except that it is stopped if $(Y_{i})$ is at
$\emptyset $ and will never visit any children of $\emptyset $
afterwards. (Note that it is not necessarily the case that $(
\overline{Y}_{i})$ is stopped on the last visit of $(Y_{i})$ to
$\Thom _{d}$, as one might expect.) A short calculation (done easily
with electrical networks) shows that whenever $(Y_{i})$ is at
$\emptyset $, the chance that it never visits any of the children of
$\emptyset $ is $(d-1)/(d^{2}+d-1)$, proving \ref{i:Td.restriction.Y}.

To prove \ref{i:Td.restriction.excursions}, we note that $(
\overline{X}_{j},E^{j}_{1},\ldots ,E^{j}_{\ell _{j}})$ is the
concatenation of either $\Geo \bigl ((d-1)/d\bigr )$ or $\Geo \bigl (d/(d+1)
\bigr )$ many $\Thom _{d}$-excursions with uniformly selected first step,
restricted to $\TT _{d}$. Since a $\Thom _{d}$-excursion restricted to
$\TT _{d}$ is a $\TT _{d}$-excursion, the only complication in finding
the distribution of $(\overline{X}_{j},E^{j}_{1},\ldots ,E^{j}_{\ell
_{j}})$ is that if the $\Thom _{d}$-excursion is from $\emptyset $ with
first step to the parent of $\emptyset $, then its restriction to
$\TT _{d}$ has length~$1$ and is effectively deleted. When $
\overline{X}_{j}\neq \emptyset $, this does not come up, taking care of
cases~\ref{i:Td.restriction.b} and \ref{i:Td.restriction.d}. When
$\overline{X}_{j}=\emptyset $, the number of excursions is thinned by
$d/(d+1)$ when $j=0$ or by $(d-1)/d$ when $j\geq 1$, since excursions
to the parent of $\emptyset $ are ignored. Thinning $\Geo (p)$ by
$q$ yields the distribution $\Geo \bigl (p / (p+(1-p)q)\bigr )$, and
applying this formula proves \ref{i:Td.restriction.a} and
\ref{i:Td.restriction.c}.
\end{proof}

The analogous lemma for $\TT ^{n}_{d}$-excursions is nearly identical,
and we omit its proof.

\begin{lemma}
\label{lem:Tnd.restriction}
Let $(X_{i})_{i\geq 0}$ and $(Y_{i})_{i\geq 0}$ be the walks on
$\Thom _{d}$ defined in Proposition~$\ref{prop:hom.decomposition}$. Let
$(\overline{X}_{i})_{i=0}^{S}$ and $(\overline{Y}_{i})_{i=0}^{T}$ be the
restrictions of these to $\TT ^{n}_{d}$. Then the following holds:
\begin{enumerate}[label=(\roman*)]%
\item
The walk $(\overline{X}_{i})_{i=0}^{S}$ is a uniform random
nonbacktracking walk on $\TT ^{d}_{n}$ stopped at the root randomly with
probability $1/(d+1)$ at time~$0$ and with probability $1/d$ at all
times past this, and stopped at the leaves randomly with probability
$d/(d+1)$ at time~$0$ and with probability one at all times past this.
\item
The walk $(\overline{Y}_{i})_{i=0}^{T}$ is a simple random walk on
$\TT _{d}^{n}$ stopped randomly at the root with probability
$(d-1)/(d^{2}+d-1)$ and at the leaves with probability $(d-1)/d$.
\item
We have
\begin{align*}
(\overline{Y}_{i})_{i=0}^{T}
&= (\overline{X}_{0},E_{1}^{0},\ldots ,E
_{\ell _{0}}^{0},\overline{X}_{1},
E_{1}^{1},\ldots ,E_{\ell _{1}}^{1},
\ldots , \overline{X}_{S},E_{1}^{S},\ldots ,E_{\ell _{S}}^{S}),
\end{align*}
where conditional on $(\overline{X}_{i})_{i=0}^{S}$ and on $S$, the
walks $(E^{j}_{i})_{i=1}^{\ell _{j}}$ are independent and distributed as
follows for $0\leq j \leq S$:
\begin{enumerate}[label=(\alph*)]%
\item
If $\overline{X}_{0}=\emptyset $, then $(\overline{X}_{0},E^{0}_{1},
\ldots ,E^{0}_{\ell _{0}})$ is the concatenation of $\Geo \bigl ((d^{2}-1)/(d
^{2}+d-1)\bigr )$ many independent $\TT ^{n}_{d}$-excursions from
$\emptyset $ with first step uniformly chosen from the children of
$\emptyset $.
\item
If $\overline{X}_{0}$ is a leaf of $\TT _{d}^{n}$, then $(
\overline{X}_{0},E^{0}_{1},\ldots ,E^{0}_{\ell _{0}})$ is the
concatenation of $\Geo \bigl ((d^{2}-1)/d^{2}\bigr )$ many independent
$\TT ^{n}_{d}$-excursions from $\overline{X}_{0}$ with first step to the
parent of $\overline{X}_{0}$.
\item
If $\overline{X}_{0}$ is neither the root nor a leaf, then $(
\overline{X}_{0},E^{0}_{1},\ldots ,E^{0}_{\ell _{0}})$ is the
concatenation of $\Geo \bigl ((d-1)/d\bigr )$ many independent
$\TT ^{n}_{d}$-excursions from $\overline{X}_{0}$ with first step chosen
uniformly from the neighbors of $\overline{X}_{0}$.
\item
If $\overline{X}_{j}=\emptyset $ for $j\geq 1$, then $(\overline{X}
_{j},E^{j}_{1},\ldots ,E^{j}_{\ell _{j}})$ is the concatenation of
$\Geo \bigl (d^{2}/(d^{2}+d-1)\bigr )$ many independent
$\TT ^{n}_{d}$-excursions from $\emptyset $ with first step chosen
uniformly from the children of $\emptyset $ other than $\overline{X}
_{j-1}$.
\item
For $j\geq 1$, if $\overline{X}_{j}$ is a leaf of $\TT _{d}^{n}$, then
$\ell _{j}=0$.
\item
For $j\geq 1$, if $\overline{X}_{j}$ is neither the root nor a leaf,
then $(\overline{X}_{j},E^{j}_{1},\ldots ,E^{j}_{\ell _{j}})$ is the
concatenation of $\Geo \bigl (d/(d+1)\bigr )$ many independent
$\TT ^{n}_{d}$-excursions from $\overline{X}_{j}$ with first step chosen
uniformly from the neighbors of $\overline{X}_{j}$ other than
$\overline{X}_{j-1}$.
\end{enumerate}
\end{enumerate}
\end{lemma}

We are nearly ready to give the decompositions of random walks on
$\TT _{d}$ and on $\TT ^{n}_{d}$. Recall the definitions of root-biased
nonbacktracking walks on $\TT _{d}$ and $\TT ^{n}_{d}$ from
Section~\ref{sec:nonbacktracking}.

\begin{proposition}
\label{prop:rw.infinite.decomposition}
Let $(X_{i})_{i\geq 0}$ be a root-biased nonbacktracking random walk
from $v_{0}$ on $\TT _{d}$, and define $(Y_{i})_{i\geq 0}$ on
$\TT _{d}$ by
%
\begin{align}
\label{eq:rw.infinite.decomposition}
(Y_{i})_{i\geq 0}
&=
(X_{0},E^{0}_{1},\ldots ,E^{0}_{\ell _{0}}, X_{1},E
^{1}_{1},\ldots ,E^{1}_{\ell _{1}},X_{2},E^{2}_{1},\ldots ,E^{2}_{\ell
_{2}},\ldots ),
\end{align}
where conditionally on $(X_{i})_{i\geq 0}$, the paths $(E_{i}^{j})_{i=1}
^{\ell _{j}}$ are independent for different $j$ and defined as follows:
\vspace*{-4pt}
\begin{enumerate}[label=(\roman*)]%
\item
Let $G\sim \Geo \bigl ((d-1)/d\bigr )$ be independent of all else. Define
$(E^{0}_{i})_{i=0}^{\ell _{0}}$ to be the concatenation of $G$
independent $\TT _{d}$-excursions with first step chosen uniformly from
the neighbors of $v_{0}$.
\label{i:inf.decomp.i}
\item
Let $G\sim \Geo \bigl (d/(d+1)\bigr )$ be independent of all else. For all
$j\geq 1$ where $X_{j}\neq \emptyset $, define $(E^{j}_{i})_{i=0}^{
\ell _{j}}$ to be the concatenation of $G$ independent
$\TT _{d}$-excursions with first step chosen uniformly from the
neighbors of $X_{j}$ other than $X_{j-1}$.%
\label{i:inf.decomp.ii}
\item
Let $G_{1}\sim \Geo \bigl (d^{2}/(d^{2}+d-1)\bigr )$ and $G_{2}\sim
\Geo \bigl ((d-1)/d\bigr )$ be independent of each other and all else. For
all $j\geq 1$ where $X_{j}=\emptyset $ and $X_{j-1}\neq X_{j+1}$, define
$(E^{j}_{i})_{i=0}^{\ell _{j}}$ as follows. With probability
$1/(d+1)$, let it be the concatenation of $G_{1}$ independent
$\TT _{d}$-excursions with first step chosen uniformly from the children
of the root other than $X_{j-1}$ followed by $G_{2}$ independent
$\TT _{d}$-excursions with first step chosen uniformly from all children
of the root. With probability $d/(d+1)$, let it just be the
concatenation of $G_{1}$ independent $\TT _{d}$-excursions with first
step chosen uniformly from the children of the root other than
$X_{j-1}$.
\label{i:inf.decomp.iii}
\item
Let $G_{1}\sim \Geo \bigl (d^{2}/(d^{2}+d-1)\bigr )$ and $G_{2}\sim
\Geo \bigl ((d-1)/d\bigr )$ be independent of each other and all else. For
all $j\geq 1$ where $X_{j}=\emptyset $ and $X_{j-1}= X_{j+1}$, define
$(E^{j}_{i})_{i=0}^{\ell _{j}}$ as the concatenation of $G_{1}$
independent $\TT _{d}$-excursions with first step chosen uniformly from
the children of the root other than $X_{j-1}$ followed by $G_{2}$
independent $\TT _{d}$-excursions with first step chosen uniformly from
all children of the root.%
\label{i:inf.decomp.iv}
\end{enumerate}
Then $(Y_{i})_{i\geq 0}$ is simple random walk on $\TT _{d}$ from
$v_{0}$.
\end{proposition}

\begin{proof}
For different values of $j$, let $\bigl ((\overline{X}^{(j)}_{i})_{i=0}
^{H_{j}},\,(\overline{Y}^{(j)}_{i})_{i=0}^{T_{j}}\bigr )$ be independent
copies of the walks $\bigl ((\overline{X}_{i})_{i=0}^{S},\,(
\overline{Y}_{i})_{i=0}^{T}\bigr )$ on $\TT _{d}$ from Lemma~$
\ref{lem:Td.restriction}$. Let $J$ be the smallest value such that
$H_{j}=T_{j}=\infty $. Concatenate $(\overline{X}^{(j)}_{i})_{i=0}
^{H_{j}}$ for $j=0,\ldots ,J$ to form $(X_{i})_{i\geq 0}$ and
concatenate $(\overline{Y}^{(j)}_{i})_{i=0}^{T_{j}}$ for $j=0,\ldots
,J$ to form $(Y_{i})_{i\geq 0}$. (This is a slight abuse of notation,
as $(X_{i})_{i\geq 0}$ and $(Y_{i})_{i\geq 0}$ are not the walks from
Proposition~$\ref{prop:hom.decomposition}$.)

We just need to show that these stitched together walks fit the
description of the statement of this proposition. We start by showing
that $(X_{i})$ is a root-biased nonbacktracking walk. From the
description of $(\overline{X}^{(j)}_{0})_{i\geq 0}$ as a randomly
stopped nonbacktracking walk given in Lemma~$\ref{lem:Td.restriction}$,
we can describe $(X_{i})$ as follows. Starting at $v_{0}$, it moves as
a uniform nonbacktracking random walk. If it arrives at the root from
a nonroot vertex, it is stopped with probability $1/d$. If this occurs,
it is reset; it forgets its previous step and moves to a uniform child
of the root. It might be stopped repeatedly before it manages to make
this move---that is, the underlying walks $(\overline{X}^{(j)}_{i})$ may
include several with $H_{j}=0$---but nonetheless the next step of
$(X_{i})$ after being reset in this way is to a uniform child of the
root.

From this description, it is clear that $X_{1}$ is uniform on the
neighbors of $v_{0}$, and that conditional on $(X_{0},\ldots ,X_{i})$
for $i\geq 1$, the distribution of $X_{i+1}$ is uniform on the neighbors
of $X_{i}$ except for $X_{i-1}$ whenever $X_{i}\neq \emptyset $. The
only question is as to the distribution of $X_{i+1}$ when
$X_{i}=\emptyset $. From our description, in this case, with probability
$(d-1)/d$ the distribution of $X_{i+1}$ is uniform on the children of
the root except for $X_{i-1}$, and with probability $1/d$ is uniform
over all children of the root. This mixture matches our definition of
a root-biased random walk.

The proof that $(Y_{i})$ is simple random walk is similar but simpler.
From Lemma~$\ref{lem:Td.restriction}$, the walk $(Y_{i})$ is the
concatenation of a sequence of independent randomly stopped simple
random walks, which is a simple random walk.

It now remains to describe the distribution of the excursions. From our
construction, $(Y_{i})$ has the form
\eqref{eq:rw.infinite.decomposition}, with one complication: The final
set of excursions of $(\overline{Y}_{i}^{(j)})_{i=0}^{T_{j}}$ is run
together with the first set of excursions of $(\overline{Y}_{i}^{(j+1)})_{i=0}
^{T_{j+1}}$. If $S_{j+1}=0$, then these excursions are also run together
with the first set in $(\overline{Y}_{i}^{(j+1)})_{i=0}^{T_{j+1}}$, and
so on. Regardless, this gives us a decomposition of $(Y_{i})$ as
\eqref{eq:rw.infinite.decomposition} with the paths $(E_{i}^{j})_{i=1}
^{\ell _{j}}$ conditionally independent given $(X_{i})$, and each made
up of geometrically many independent $\TT _{d}$-excursions. We just need
to show that parameters of the geometric distributions and the
distributions of their first steps match the statement of this lemma.
When $X_{j}\neq \emptyset $, the path $(E^{j}_{1},\ldots ,E^{j}_{\ell
_{j}})$ is taken from a single $(\overline{Y}_{i}^{(j)})$, and from
Lemma~\ref{lem:Td.restriction}\ref{i:Td.restriction.excursions}, the path is
distributed as given in \ref{i:inf.decomp.i} or \ref{i:inf.decomp.ii},
depending on whether $j=0$.

When $X_{j}=\emptyset $, we first consider the case that $j=0$. Then the
path $(E^{0}_{1},\ldots ,E^{0}_{\ell _{0}})$ is the concatenation of the
excursions in $(\overline{Y}^{(0)}_{i})_{i=0}^{T_{0}},\ldots ,(
\overline{Y}^{(J-1)}_{i})_{i=0}^{T_{J-1}}$ along with the first set of
excursions in $(\overline{Y}^{(J)}_{i})_{i=0}^{T_{J}}$, where $J$ is the
smallest value so that $H_{j}\geq 1$. Note that $J-1$ is then
distributed as $\Geo \bigl (d/(d+1)\bigr )$, since for each $j$, the
events $H_{j}\geq 1$ occur independently with probability $d/(d+1)$.
Each set of excursions is the concatenation of $\Geo \bigl ((d^{2}-1)/(d
^{2}+d-1)\bigr )$ many independent $\TT _{d}$-excursions with first step
chosen uniformly from the children of the root. Thus, the total number
of excursions is distributed as the sum of $1+\Geo \bigl (d/(d+1)
\bigr )$ many independent
$\Geo \bigl ((d^{2}-1)/(d^{2}+d-1)\bigr )$-distributed random variables,
which is $\Geo \bigl ((d-1)/d\bigr )$. (Here we use the general fact that
a sum of $1+\Geo (p)$ many independent $\Geo (q)$ random variables is
distributed as $\Geo \bigl (pq/(1-q+pq)\bigr )$. This confirms that
$(E^{0}_{1},\ldots ,E^{0}_{\ell _{0}})$ is distributed as
\ref{i:inf.decomp.i} when $X_{0}=\emptyset $.

If $X_{j}=\emptyset $ for $j\geq 1$ and $X_{j-1}=X_{j+1}$, then the
underlying system of stopped and restarted uniform nonbacktracking
random walks moved from $X_{j-1}$ to $X_{j}$, was stopped, and then
moved to $X_{j-1}$ when restarted. Thus, $(E^{j}_{i})_{i=1}^{\ell _{j}}$
is the concatenation of two collections of excursions: first,
$\Geo \bigl (d^{2}/(d^{2}+d-1)\bigr )$ many independent
$\TT _{d}$-excursions from $\emptyset $ with first step chosen uniformly
from the children of $\emptyset $ other than $X_{j-1}$, by
Lemma~\ref{lem:Td.restriction}\ref{i:Td.restriction.excursions}\ref{i:Td.restriction.c};
second, $1+\Geo \bigl (d/(d+1)\bigr )$ many concatenated independent
$\Geo \bigl ((d^{2}-1)/(d^{2}+d-1)\bigr )$ many $\TT _{d}$-excursions with
first step chosen uniformly from the children of the root, as in the
case $j=0$. Together, this matches the description of
\ref{i:inf.decomp.iv} from this proposition.

Last, suppose that $X_{j}=\emptyset $ for $j\geq 1$ and $X_{j-1}
\neq X_{j+1}$. Now, it is possible that the underlying system of walks
was restarted at the root or not. Conditioning on $X_{j-1}\neq X_{j+1}$
makes the probability that we were stopped $1/(d+1)$. If so, then the
path $(E^{j}_{i})_{i=1}^{\ell _{j}}$ is distributed as in
\ref{i:inf.decomp.iv}. Otherwise, it is the concatenation of
$\Geo \bigl (d^{2}/(d^{2}+d-1)\bigr )$ many independent
$\TT _{d}$-excursions from $\emptyset $ with first step chosen uniformly
from the children of $\emptyset $ other than $X_{j-1}$. This shows that
$(E^{j}_{i})_{i=1}^{\ell _{j}}$ is distributed as given in
\ref{i:inf.decomp.iii}.
\end{proof}

Last, we give the decomposition for walks on $\TT ^{n}_{d}$. The proof
is very similar to that of Proposition~$
\ref{prop:rw.infinite.decomposition}$, and we omit it.
%
\begin{proposition}
\label{prop:rw.decomposition}
Let $(X_{i})_{i\geq 0}$ be a root-biased nonbacktracking random walk
from $u_{0}$ on $\TT _{d}^{n}$, and define another path $(Y_{i})_{i
\geq 0}$ on $\TT _{d}^{n}$ by
\begin{align*}
(Y_{i})_{i\geq 0}
&=
(X_{0},E^{0}_{1},\ldots ,E^{0}_{\ell _{0}}, X_{1},E
^{1}_{1},\ldots ,E^{1}_{\ell _{1}},X_{2},E^{2}_{1},\ldots ,E^{2}_{\ell
_{2}},\ldots ),
\end{align*}
where conditionally on $(X_{i})_{i\geq 0}$, the paths $(E_{i}^{j})_{i=1}
^{\ell _{j}}$ are independent for different $j$ and defined as in Proposition~$
\ref{prop:rw.infinite.decomposition}$, except that all
$\TT _{d}$-excursions are replaced by $\TT _{d}^{n}$-excursions, and
when $X_{j}$ is a leaf, $(E^{j}_{i})_{i=1}^{\ell _{j}}$ is distributed
is the concatenation of $\Geo \bigl ((d-1)/d\bigr )$ independent
$\TT _{d}^{n}$-excursions with first step to the parent of $X_{j}$. Then
$(Y_{i})_{i\geq 0}$ is a simple random walk on $\TT ^{d}_{n}$ from
$u_{0}$.
\end{proposition}

\begin{corollary}
\label{cor:walk.coupling}
Let $(X_{i})_{i\geq 0}$ be a root-biased nonbacktracking random walk
from $x_{0}$, and let $(Y_{i})_{i\geq 0}$ be a simple random walk from
$x_{0}$, both on $\TT _{d}$ or both on $\TT ^{n}_{d}$. Suppose that they
are coupled as in Proposition~\ref{prop:rw.infinite.decomposition} or
\ref{prop:rw.decomposition}. Recall that $\ell _{j}$ is the length of the
path inserted in between $X_{j}$ and $X_{j+1}$ to form $(Y_{i})_{i
\geq 0}$. Then the random variables $(\ell _{i})_{i\geq 0}$ are
independent conditional on $(X_{i})_{i\geq 0}$, and it holds for some
absolute constant $c>0$ and all real numbers $t\geq 0$ that
%
\begin{align}
\label{eq:walk.coupling}
\P \bigl [\ell _{j}\geq t+2\bigmid (X_{i})_{i\geq 0}\bigr ]\leq e^{-ct}.
\end{align}
\end{corollary}
\begin{proof}
Let $U$ be distributed as the length of a $\Thom _{d}$-excursion, which
stochastically dominates the length of a $\TT _{d}$- or
$\TT ^{n}_{d}$-excursion. We can characterize $U$ as follows. Let
$(S_{k})_{k\geq 1}$ be a biased random walk that moves left with
probability $d/(d+1)$ and right with probability $1/(d+1)$, with
$H_{1}=1$. Then $U$ is the first time that $(S_{k})$ hits $0$, which we
note is always even. If $U\geq 2+2u$, then at least $u$ of the first
$2u$ steps are to the right, which by Proposition~$
\ref{prop:Poi.bound}$ has probability bounded by
\begin{align*}
\P \Bigl [\Bin \bigl (2u, \tfrac{1}{d+1}\bigr ) \geq u\Bigr ]
&\leq
\exp \biggl ( -\frac{(d-1)u}{(d+1)\bigl (\frac{2}{3} +\frac{4}{d-1}
\bigr )}\biggr )
\leq e^{-u/14}.
\end{align*}
Hence $U/2\preceq 1 + \Geo \bigl (1-e^{-1/14}\bigr )$.

By Proposition~\ref{prop:rw.infinite.decomposition} or
\ref{prop:rw.decomposition}, conditional on $(X_{i})_{i\geq 0}$ the path
$X_{j},E^{j}_{1},\ldots ,E^{j}_{\ell _{j}}$ is made up of stochastically
at most $\Geo \bigl ((d-1)/d\bigr )$ concatenated $\TT _{d}$- or
$\TT ^{n}_{d}$-excursions. We work conditionally on $(X_{i})_{i\geq 0}$
for the rest of the proof. The random variable $\ell _{j}$ is
stochastically dominated by twice a sum of $1+\Geo \bigl ((d-1)/d
\bigr )$ many independent random variables distributed as $1 + \Geo
\bigl (1-e^{-1/14}\bigr )$. As the sum of $1+\Geo (p)$ many independent
random variables distributed as $1+\Geo (q)$ is distributed as
$1+\Geo (pq)$,
\begin{align*}
\ell _{j} \preceq 2 + 2\Geo \biggl ( \frac{d-1}{d}\bigl (1-e^{-1/14}
\bigr ) \biggr )
\preceq 2 + 2\Geo \biggl (\frac{1-e^{-1/14}}{2}\biggr ).
\end{align*}
This shows that for all $t\geq 0$
\begin{align*}
\P \bigl [\ell _{j}\geq t+2\bigmid (X_{i})_{i\geq 0}\bigr ]
&\leq \biggl (\frac{1+e^{-1/14}}{2}\biggr )^{t/2}.
\end{align*}
Hence \eqref{eq:walk.coupling} is satisfied. The conditional
independence of $(\ell _{i})_{i\geq 0}$ follows directly from
Proposition~\ref{prop:rw.infinite.decomposition} or
\ref{prop:rw.decomposition}.
\end{proof}

\begin{proof}[Proof of Proposition~$\ref{cor:time.change}$]
Let $\Vv $ be the set of vertices visited in the nonbacktracking model
$(\eta ,S')$ by time $t$. Suppose $v\in \Vv $. This means that there
exists some sequence of frogs in $(\eta ,S')$ starting with the initial
one such that each visits the next and the last visits $v$. Consider the
path of length at most $t$ formed by pasting together the portion of
each frog's walk up until it hits the next frog or hits $v$. The same
frogs take these same steps in the model $(\eta ,S)$, but with
excursions inserted between each step. By Corollary~$
\ref{cor:walk.coupling}$, conditional on $(\eta ,S')$, these excursions
are independent and their lengths have an exponential tail. Choosing
$C_{0}$ depending on $b$ and applying Proposition~$
\ref{prop:exp.sum.bound}$, conditional on $(\eta ,S')$, the combined
length of these excursions is at most $C_{0}t$ with probability at least
$1-e^{-2bt}$. Let $C=1+C_{0}$. We have shown that if $E(v)$ is the event
that $v$ is unvisited in $(\eta ,S)$ at time $Ct$, then
\begin{align*}
\P \bigl [ E(v) \bigmid (\eta ,S') \bigr ]
&\leq e^{-2bt}
\end{align*}
for any $v$ visited in $(\eta ,S')$ by time~$t$.

Observe that $\abs{\Vv }\leq d^{t}$, since no vertex beyond level~$t$
can be visited by time~$t$. By a union bound,
\begin{align*}
\P \Biggl [ \bigcup _{v\in \Vv } E(v) \Biggmid (\eta ,S') \Biggr ]
&
\leq d^{t}e^{-2bt}\leq e^{-bt},
\end{align*}
since $b\geq \log d$. Taking expectations completes the proof.
\end{proof}

\section{Proof of Proposition~$\ref{prop:sequence_hard}$}%
\label{sec:painful.appendix}
In this appendix, we prove Proposition~$\ref{prop:sequence_hard}$.
First, we give some notation representing weighted averages that we will
use throughout. Given nonnegative weights $w_{1},\ldots ,w_{n}$, not all
zero, and quantities $a_{1},\ldots ,a_{n}$, we define
\begin{align*}
\WA (w_{1},a_{1}\sep \ldots \sep w_{n},a_{n})
&= \frac{w_{1}a_{1}+
\cdots +w_{n}a_{n}}{w_{1}+\cdots +w_{n}},
\end{align*}
the weighted average of $a_{1},\ldots ,a_{n}$ with weights proportional
to $w_{1},\ldots ,w_{n}$. The following lemma is easy to check by direct
calculation.
%
\begin{lemma}
\label{lem:reduction}
For any $0\leq k \leq n-1$,
\begin{align*}
\WA (w_{1},a_{1}\sep \ldots \sep w_{n},a_{n})
&= \WA (w_{1},a_{1}\sep
\cdots \sep w_{k},a_{k}\sep w_{*},a_{*}),
\end{align*}
where
\begin{align*}
w_{*}
&= w_{k+1}+\cdots +w_{n},
\\
a_{*}
&= \WA (w_{k+1},a_{k+1}\sep \ldots \sep w_{n},a_{n}).
\end{align*}
\end{lemma}

Let $p_{n}=e^{-\lambda _{n}/d}$. Proving that $(\lambda _{n})_{n\geq 1}$
is decreasing is equivalent to proving that $(p_{n})_{n\geq 1}$ is
increasing, which is what we will do. Also let $p_{0}=a$, which is
defined in Definition~$\ref{def:sequences}$ as $a=e^{-\mu /(d+1)}$.
Equation~\eqref{eq:lambda.link} tells us that $P_{n}=p_{1}\cdots p
_{n}$. Also define $P_{0}=1$. We can then recast \eqref{eq:P_n} as
\begin{align*}
P_{n+1}
&= p_{0}^{1/d}\Biggl [\sum _{i=0}^{n-1} (p_{0}\cdots p_{i-1})(1-p
_{i})(p_{1}\cdots p_{n-i}) + p_{0}\cdots p_{n-1}
\Biggr ]
\end{align*}
for $n\geq 2$. Since $P_{n+1}/P_{n}=p_{n+1}$, we obtain
%
\begin{align}
\label{eq:p_n}
p_{n+1}
&= \frac{\sum _{i=0}^{n-1} (p_{0}\cdots p_{i-1})(1-p_{i})(p
_{1}\cdots p_{n-i}) + p_{0}\cdots p_{n-1}}{\sum _{i=0}^{n-2} (p_{0}
\cdots p_{i-1})(1-p_{i})(p_{1}\cdots p_{n-1-i}) + p_{0}\cdots p_{n-2}}
\end{align}
for $n\geq 2$.
%
\begin{lemma}
\label{lem:suf1}
Define
\begin{align*}
t_{i}
&= (P_{i-1}-P_{i})P_{n-1-i},
\quad \quad
1\leq i\leq n-2,
\\
t_{n-1}
&= P_{n-2},
\\
q
&= (1-p_{n-1})p_{1} + p_{n-1} = p_{1} + (1-p_{1})p_{n-1}.
\end{align*}
For $n\geq 2$, if
%
\begin{align}
\label{eq:suf1}
p_{n}
&\leq \WA (t_{1}, p_{n-1}\sep t_{2}, p_{n-2}\sep \ldots \sep t
_{n-2}, p_{2}\sep t_{n-1}, q),
\end{align}
then $p_{n}\leq p_{n+1}$.
\end{lemma}
\begin{proof}
Expanding out the right-hand side of \eqref{eq:suf1} gives our
assumption as
\begin{align*}
p_{n}
&\leq \frac{\sum _{i=1}^{n-1} (P_{i-1}-P_{i})P_{n-i} + P_{n-1}}{
\sum _{i=1}^{n-2} (P_{i-1}-P_{i})P_{n-1-i} + P_{n-2}}
\\
&= \frac{p_{0}\Bigl [\sum _{i=1}^{n-1} (P_{i-1}-P_{i})P_{n-i} + P_{n-1}
\Bigr ]}{p_{0}\Bigl [\sum _{i=1}^{n-2} (P_{i-1}-P_{i})P_{n-1-i} + P_{n-2}
\Bigr ]}.
\end{align*}
For any positive $A,B,C,D$, if $x\leq B/C$, then $x \leq (xA + B) /
(A+C)$. Applying this with $x=p_{n}$, we have
\begin{align*}
p_{n}
&\leq \frac{p_{n}(1-p_{0})P_{n-1}+p_{0}\Bigl [\sum _{i=1}^{n-1} (P
_{i-1}-P_{i})P_{n-i} + P_{n-1}\Bigr ]}{(1-p_{0})P_{n-1}+p_{0}\Bigl [
\sum _{i=1}^{n-2} (P_{i-1}-P_{i})P_{n-1-i} + P_{n-2}\Bigr ]}
\\
&=\frac{P_{n+1}}{P_{n}} = p_{n+1}.\qedhere
\end{align*}
\end{proof}

Our goal now is to show that \eqref{eq:suf1} holds. In
\eqref{eq:two.observations.1}, we will express $p_{n}$ as a weighted
average. Using Lemma~$\ref{lem:i.i+1}$, we then change this expression
one term at a time to create a chain of inequalities that ends with
$p_{n+1}$. We now define expressions for forming this chain.

\begin{definition}
For fixed $n$, we define functions $u_{i}(x_{1},\ldots ,x_{n-1})$ and
related quantities $u_{i}^{j}$. Let
\begin{align*}
u_{i}(x_{1},\ldots ,x_{n-2})
&= x_{1}\cdots x_{i-1}(1-x_{i}) P_{n-1-i},
\quad \quad
1\leq i\leq n-2,
\\
u_{n-1}(x_{1},\ldots ,x_{n-2})
&= x_{1}\dots x_{n-2}.
\end{align*}
For all $1\leq i \leq n-1$, let
\begin{align*}
u_{i}^{0}
&= u_{i}(p_{0},p_{1},p_{2},\ldots ,p_{n-3}),
\\
u_{i}^{1}
&= u_{i}(p_{1},p_{1},p_{2},\ldots ,p_{n-3}),
\\
u_{i}^{j}
&= u_{i}( p_{1} , \hdots , p_{j-1}, p_{j}, p_{j}, p_{j+1},
\hdots , p_{n-3} )
\qquad
\text{ for } 1 < j < n-2,
\\
u_{i}^{n-2}
&= u_{i}(p_{1},p_{2},p_{3},\ldots ,p_{n-2}).
\end{align*}
\end{definition}

\begin{lemma}
\label{lem:two.observations}
Let $q'=p_{1} +(1-p_{1})p_{n-2}$, and recall the definition of $q$ from
Lemma~$\ref{lem:suf1}$. We have
%
\begin{align}
\label{eq:two.observations.1}
p_{n}
&= \WA ( u_{1}^{0},p_{n-1}\sep u_{2}^{0},p_{n-2}\sep \ldots \sep u
_{n-2}^{0},p_{2}\sep u_{n-1}^{0},q' ),
\end{align}
and
%
\begin{align}
\label{eq:two.observations.2}
\begin{split}
&\WA (t_{1}, p_{n-1}\sep t_{2}, p_{n-2}\sep \ldots \sep t_{n-2}, p
_{2}\sep t_{n-1}, q)
\\
&
\qquad
\qquad
\qquad
\qquad
=
\WA (u_{1}^{n-2},p_{n-1}\sep u_{2}^{n-2},p_{n-2}\sep \ldots \sep u
_{n-2}^{n-2},p_{2}\sep u_{n-1}^{n-2},q).
\end{split}
\end{align}
Thus, proving \eqref{eq:suf1} is equivalent to proving that
%
\begin{align}
\label{eq:suf2}
\begin{split}
&\WA ( u_{1}^{0},p_{n-1}\sep u_{2}^{0},p_{n-2}\sep \ldots \sep u_{n-2}
^{0},p_{2}\sep u_{n-1}^{0},q' )
\\
&
\qquad
\qquad
\qquad
\qquad
\leq \WA (u_{1}^{n-2},p_{n-1}\sep u_{2}^{n-2},p_{n-2}\sep \ldots \sep u
_{n-2}^{n-2},p_{2}\sep u_{n-1}^{n-2},q),
\end{split}
\end{align}
\end{lemma}
\begin{proof}
For \eqref{eq:two.observations.1}, observe that from \eqref{eq:p_n},
\begin{align*}
p_{n}
&= \frac{\sum _{i=1}^{n-1} (p_{0}\cdots p_{i-2})(1-p_{i-1})(p
_{1}\cdots p_{n-i}) + p_{0}\cdots p_{n-2}}{\sum _{i=1}^{n-2} (p_{0}
\cdots p_{i-2})(1-p_{i-1})(p_{1}\cdots p_{n-1-i}) + p_{0}\cdots p_{n-3}}
\\
&= \frac{\sum _{i=1}^{n-2} u_{i}^{0}p_{n-i} +u_{n-1}^{0}\bigl (p_{1} +(1-p
_{1})p_{n-2}\bigr )}{\sum _{i=1}^{n-2} u_{i}^{0} + u_{n-1}^{0}}
\\
&= \WA ( u_{1}^{0},p_{n-1}\sep u_{2}^{0},p_{n-2}\sep \ldots \sep u
_{n-2}^{0},p_{2}\sep u_{n-1}^{0},q' ),
\end{align*}
For \eqref{eq:two.observations.2}, note that $u_{i}^{n-2}=t_{i}$.
\end{proof}

\begin{lemma}
\label{lem:i.i+1}
Fix $n\geq 2$, and assume that $0<p_{0}\leq \cdots \leq p_{n}$. For
$0\leq i\leq n-3$,
%
\begin{align}
\label{eq:i.i+1}
\begin{split}
&\WA ( u_{1}^{i},p_{n-1}\sep u_{2}^{i},p_{n-2}\sep \cdots \sep u_{n-2}
^{i},p_{2}\sep u_{n-1}^{i},q)
\\
&
\qquad
\qquad
\qquad
\qquad
\leq \WA ( u_{1}^{i+1},p_{n-1}\sep u_{2}^{i+1},p_{n-2}\sep \cdots \sep u
_{n-2}^{i+1},p_{2}\sep u_{n-1}^{i+1},q).
\end{split}
\end{align}
\end{lemma}

\begin{proof}
We prove this by induction on $i$. Assume that \eqref{eq:i.i+1} holds
with $i$ replaced by $j$, whenever $0\leq j< i$. Now, our goal is to
show that it holds for $i$. Noting that $u^{i}_{i+1}=(1-p_{i})P_{i}P
_{n-i-2}$, we apply Lemma~$\ref{lem:reduction}$ to obtain
%
\begin{align}
\label{eq:reduced}
\begin{split}
&\WA ( u_{1}^{i},p_{n-1}\sep u_{2}^{i},p_{n-2}\sep \cdots \sep u_{n-2}
^{i},p_{2}\sep u_{n-1}^{i},q)
\\
&
\qquad
\qquad
\qquad
= \WA ( u_{1}^{i},p_{n-1}\sep u_{2}^{i},p_{n-2}\sep \cdots \sep u^{i}
_{i},p_{n-i}\sep (1-p_{i})P_{i}P_{n-i-2},\,p_{n-i-1}\sep p_{i}u_{*}, p
_{*}),
\end{split}
\end{align}
where
%
\begin{align}
u_{*}
&= \frac{1}{p_{i}}\sum _{k=i+2}^{n-1} u^{i}_{k} =
\sum _{k=i+2}
^{n-2}P_{k-2}(1-p_{k-1})P_{n-1-k} + P_{n-3},
\\
p_{*}
&= \WA ( u_{i+2}^{i},p_{n-i-2}\sep \cdots \sep u_{n-2}^{i},p
_{2}\sep u_{n-1}^{i},q ).
\label{eq:p*.def}
\end{align}
We claim that $p_{*} \geq p_{n-1}$. Indeed, suppose that $p_{*} < p
_{n-1}$. As we are given that $p_{n-i-1}\leq \cdots \leq p_{n-1}$, the
right hand side of \eqref{eq:reduced} is a weighted average of terms of
size at most $p_{n-1}$, with the last one, $p_{*}$, strictly smaller
than $p_{n-1}$. As the weight on $p_{*}$ is strictly positive, this
implies that \eqref{eq:reduced} is strictly smaller than $p_{n-1}$.
Recall that $q' = p_{1} + (1-p_{1})p_{n-2}$ and $q=p_{1} + (1-p_{1})p
_{n-1}$. Since $p_{n-2}\leq p_{n-1}$ by assumption, we have
$q'\leq q$. Using this, the assumption that $p_{n-1}\leq p_{n}$, and
\eqref{eq:two.observations.1},
\begin{align*}
p_{n-1}\leq p_{n}
&= \WA ( u_{1}^{0},p_{n-1}\sep u_{2}^{0},p_{n-2}\sep
\ldots \sep u_{n-2}^{0},p_{2}\sep u_{n-1}^{0},q' )
\\
&\leq \WA ( u_{1}^{0},p_{n-1}\sep u_{2}^{0},p_{n-2}\sep \ldots \sep u
_{n-2}^{0},p_{2}\sep u_{n-1}^{0},q ).
\end{align*}
Applying our inductive hypothesis, we obtain
\begin{align*}
p_{n-1}
&\leq \WA ( u_{1}^{i},p_{n-1}\sep u_{2}^{i},p_{n-2}\sep \ldots
\sep u_{n-2}^{i},p_{2}\sep u_{n-1}^{i},q ).
\end{align*}
But this is a contradiction, since \eqref{eq:reduced} is bounded from
below by $p_{n-1}$. Hence $p_{*}\geq p_{n-1}$.

Thus, the right-hand side of \eqref{eq:reduced} is a weighted sum of
terms of which $p_{*}$ and $p_{n-i-1}$ are respectively the largest and
smallest. Increasing the weight on the largest term and decreasing it
on the smallest term can only make the sum larger. Since $p_{i}\leq p
_{i+1}$,
%
\begin{align}
\label{eq:sq2}
\begin{split}
&\WA ( u_{1}^{i},p_{n-1}\sep u_{2}^{i},p_{n-2}\sep \cdots \sep u_{i}
^{i},p_{n-i}\sep (1-p_{i})P_{i}P_{n-i-2},p_{n-i-1}\sep p_{i}u_{*}, p
_{*})
\\
&
\qquad
\qquad
\leq \WA ( u_{1}^{i},p_{n-1}\sep u_{2}^{i},p_{n-2}\sep \cdots \sep u
_{i}^{i},p_{n-i}\sep (1-p_{i+1})P_{i}P_{n-i-2},p_{n-i-1}\sep p_{i+1}u
_{*}, p_{*}).
\end{split}
\end{align}
Now, we note that
\begin{align*}
u_{k}^{i}
&=u_{k}^{i+1} \text{ for $1\leq k\leq i$,}
\\
u_{i+1}^{i+1}
&=(1-p_{i+1})P_{i}P_{n-i-2},
\\
\intertext{and}
p_{i+1}u_{*}
&= \sum _{k=i+2}^{n-1} u_{k}^{i+1}.
\end{align*}
Hence,
%
\begin{align}
\label{eq:sq3}
\begin{split}
&\WA ( u_{1}^{i},p_{n-1}\sep u_{2}^{i},p_{n-2}\sep \cdots \sep u_{i}
^{i},p_{n-i}\sep (1-p_{i+1})P_{i}P_{n-i-2},p_{n-i-1}\sep p_{i+1}u_{*},
p_{*})
\\
&
\qquad
\qquad
=
\WA ( u_{1}^{i+1},p_{n-1}\sep u_{2}^{i+1},p_{n-2}\sep \cdots \sep u
_{i}^{i+1},p_{n-i}\sep u_{i+1}^{i+1},p_{n-i-1}\sep p_{i+1}u_{*}, p
_{*}).
\end{split}
\end{align}
Recall the definition of $p_{*}$ in \eqref{eq:p*.def}, and observe that
$u_{k}^{i+1}=p_{i+1}u_{k}^{i}/p_{i}$ for $k\geq n+2$. Since rescaling
all weights by the same factor does not change the weighted sum,
\begin{align*}
p_{*}
&= \WA \biggl (\frac{p_{i+1}}{p_{i}} u_{i+2}^{i},\,p_{n-i-2}\biggsep
\cdots \biggsep
       \frac{p_{i+1}}{p_{i}}u_{n-2}^{i},\,p_{2}\biggsep
       \frac{p_{i+1}}{p_{i}}u_{n-1}^{i},\,q \biggr )
\\
&= \WA \bigl ( u_{i+2}^{i+1},\,p_{n-i-2}\bigsep \cdots \bigsep u_{n-2}
^{i+1},\,p_{2}\bigsep u_{n-1}^{i+1},\,q \bigr ).
\end{align*}
Applying Lemma~$\ref{lem:reduction}$ in reverse,
%
\begin{align}
\label{eq:sq4}
\begin{split}
&\WA ( u_{1}^{i+1},p_{n-1}\sep u_{2}^{i+1},p_{n-2}\sep \cdots \sep u
_{i}^{i+1},p_{n-i}\sep u_{i+1}^{i+1},p_{n-i-1}\sep p_{i+1}u_{*}, p
_{*})
\\
&
\qquad
\qquad
= \WA (u_{1}^{i+1}, p_{n-1} \sep u_{2}^{i+1},p_{n-2}\sep \cdots \sep
      u_{n-2}^{i+1}, p_{2}\sep u_{n-1}^{i+1}, q).
\end{split}
\end{align}
Together, \eqref{eq:reduced}, \eqref{eq:sq2}, \eqref{eq:sq3}, and
\eqref{eq:sq4} show that
\begin{align*}
&\WA ( u_{1}^{i},p_{n-1}\sep u_{2}^{i},p_{n-2}\sep \cdots \sep u_{n-2}
^{i},p_{2}\sep u_{n-1}^{i},q)
\\
&
\qquad
\qquad
\qquad
\leq \WA (u_{1}^{i+1}, p_{n-1} \sep u_{2}^{i+1},p_{n-2}\sep \cdots \sep
      u_{n-2}^{i+1}, p_{2}\sep u_{n-1}^{i+1}, q).
\end{align*}
This completes the induction, proving that \eqref{eq:i.i+1} holds for
all $0\leq i\leq n-3$.
\end{proof}

\begin{proof}[Proof of Proposition~$\ref{prop:sequence_hard}$]
We prove that $p_{0}\leq p_{1}\leq \cdots $ by induction. We have
\begin{align*}
p_{0} = e^{-\mu /(d+1)} \leq e^{-\mu /d(d+1)} = p_{1}.
\end{align*}
From \eqref{eq:P_2},
\begin{align*}
p_{2}
&= \frac{P_{2}}{P_{1}} = (1-p_{0})p_{1} + p_{0}\geq p_{1}.
\end{align*}

Now, suppose that $p_{0}\leq \cdots \leq p_{n}$ for some $n\geq 2$.
Recall that $q' = p_{1} + (1-p_{1})p_{n-2}$ and $q=p_{1} + (1-p_{1})p
_{n-1}$, and hence $q'\leq q$. Therefore,
\begin{align*}
&\WA ( u_{1}^{0},p_{n-1}\sep u_{2}^{0},p_{n-2}\sep \ldots \sep u_{n-2}
^{0},p_{2}\sep u_{n-1}^{0},q' )
\\
&
\qquad
\qquad
\qquad
\qquad
\leq \WA ( u_{1}^{0},p_{n-1}\sep u_{2}^{0},p_{n-2}\sep \ldots \sep u
_{n-2}^{0},p_{2}\sep u_{n-1}^{0},q).
\end{align*}
By Lemma~$\ref{lem:i.i+1}$,
\begin{align*}
&\WA ( u_{1}^{0},p_{n-1}\sep u_{2}^{0},p_{n-2}\sep \ldots \sep u_{n-2}
^{0},p_{2}\sep u_{n-1}^{0},q )
\\
&
\qquad
\qquad
\qquad
\qquad
\leq \WA ( u_{1}^{n-2},p_{n-1}\sep u_{2}^{n-2},p_{n-2}\sep \ldots \sep u
_{n-2}^{n-2},p_{2}\sep u_{n-1}^{n-2},q ).
\end{align*}
This proves \eqref{eq:suf2}, which by Lemma~$
\ref{lem:two.observations}$ is equivalent to proving \eqref{eq:suf1}.
By Lemma~$\ref{lem:suf1}$, this shows that $p_{n}\leq p_{n+1}$,
completing the induction to prove that $(p_{n})_{n\geq 0}$ is
increasing, which is equivalent to $(\lambda _{n})_{n\geq 1}$ being
decreasing.
\end{proof}

\section{Miscellaneous concentration inequalities}%
\label{sec:concentration}
%
\begin{proposition}
\label{prop:Poi.bound}
Let $\E Y=\lambda $, and suppose either that $Y$ is Poisson or that
$Y$ is a sum of independent random variables supported on $[0,1]$. For
any $0<\alpha <1$,
\begin{align*}
\P [Y\leq \alpha \lambda ]
&\leq \exp \biggl (-\frac{(1-\alpha )^{2}
\lambda }{2}\biggr ),
\end{align*}
and for any $\alpha >1$,
\begin{align*}
\P [Y\geq \alpha \lambda ]
&\leq \exp \Biggl (-\frac{(\alpha -1)\lambda
}{\frac{2}{3} + \frac{2}{\alpha -1}}\Biggr ).
\end{align*}
\end{proposition}
\begin{proof}
These inequalities are well known consequences of the
Cram\'{e}r--Chernoff method of bounding the moment generating function
and applying Markov's inequality (see \cite[Section~2.2]{BLM}),
though it is difficult to find these exact bounds in the literature. For
a convenient statement of these inequalities obtained by other means,
apply \cite[Theorem~3.3]{CGJ} with $c=p=1$.
\end{proof}

\begin{proposition}
\label{prop:exp.sum.bound}
Let $(X_{i})_{i=1}^{n}$ be a collection of independent random variables
satisfying
\begin{align*}
\P [X_{i} \geq \ell ] \leq Ce^{-b\ell }
\end{align*}
for some $C$ and $b>0$ and all $\ell \geq 1$. Then for any $b'>0$, there
exists $C'$ depending on $C$, $b$, and $b'$ such that
\begin{align*}
\P \Biggl [\sum _{i=1}^{n} X_{i} \geq C'n \Biggr ]
&\leq e^{-b'n}.
\end{align*}
We can take $C'=2(b'+C)/b$.
\end{proposition}
\begin{proof}
This is another consequence of the Cram\'{e}r--Chernoff method. Observe
that
\begin{align*}
\E e^{sX_{i}}
&= \int _{0}^{\infty } \P [e^{sX_{i}}\geq t]\,dt
\\
&\leq 1 + \int _{1}^{\infty }\P [X_{i}\geq \log t/s]\,dt
\\
&\leq 1 + \int _{1}^{\infty } C t^{-b/s}\,dt = 1 + \frac{Cs}{b-s}
\end{align*}
for any $0\leq s<b$. Hence
\begin{align*}
\P \Biggl [\sum _{i=1}^{n} X_{i} > C'n \Biggr ]
&\leq e^{-sC'n}\biggl (1+
\frac{Cs}{b-s}\biggr )^{n}
\\
&\leq \exp \Biggl [-n\biggl (sC' - \frac{Cs}{b-s}\biggr )
\Biggr ].
\end{align*}
Rather than optimizing in $s$, we choose $s=b/2$ for the sake of
simplicity and convenience, obtaining
\begin{align*}
\P \Biggl [\sum _{i=1}^{n} X_{i} > C'n \Biggr ]
&\leq e^{-sC'n}\biggl (1+
\frac{Cs}{b-s}\biggr )^{n}
\\
&\leq \exp \Biggl [-n\biggl (\frac{bC'}{2} - C\biggr )
\Biggr ] \leq e
^{-b'n}
\end{align*}
when $C'\geq 2(b'+C)/b$.
\end{proof}



\begin{thebibliography}{99}

\bibitem{shape}
O.~S.~M. Alves, F.~P. Machado, and S.~Yu. Popov, \emph{The shape theorem
for the frog model}, Ann.\ Appl.\ Probab. \textbf{12} (2002), no.~2,
533--546.
\MR{1910638}

\bibitem{random_shape}
O.~S.~M. Alves, F.~P. Machado, S.~Yu. Popov, and K.~Ravishankar,
\emph{The shape theorem for the frog model with random initial configuration},
Markov Process. Related Fields \textbf{7} (2001), no.~4, 525--539. \MR{1893139 (2003f:60171)}

\bibitem{bmf}
Erin Beckman, Emily Dinan, Rick Durrett, Ran Huo, and Matthew Junge,
\emph{Asymptotic behavior of the brownian frog model}, available at
\ARXIV{1710.05811}, 2017.
\MR{3870447}

\bibitem{BLM}
St\'{e}phane Boucheron, G\'{a}bor Lugosi, and Pascal Massart,
\emph{Concentration inequalities}, Oxford University Press, Oxford,
2013, A nonasymptotic theory of independence, With a foreword by Michel
Ledoux.
\MR{3185193}

\bibitem{CGJ}
Nicholas Cook, Larry Goldstein, and Tobias Johnson, \emph{Size biased
couplings and the spectral gap for random regular graphs}, Ann. Probab.
\textbf{46} (2018), no.~1, 72--125.
\MR{3758727}

\bibitem{nina_drift2}
Christian D\"{o}bler, Nina Gantert, Thomas H\"{o}felsauer, Serguei
Popov, and Felizitas Weidner, \emph{Recurrence and Transience of Frogs
with Drift on $\mathbb{Z}^{d}$}, available at \ARXIV{1709.00038}, 2017.
\MR{3858916}

\bibitem{dobler_drift}
Christian D\"{o}bler and Lorenz Pfeifroth, \emph{Recurrence for the frog
model with drift on $\mathbb{Z}^{d}$}, Electron. Commun. Probab.
\textbf{19} (2014), no. 79, 13.
\MR{3283610}

\bibitem{nina_drift1}
Nina Gantert and Philipp Schmidt, \emph{Recurrence for the frog model
with drift on $\mathbb{Z}$}, Markov Process. Related Fields
\textbf{15} (2009), no.~1, 51--58.
\MR{2509423}

\bibitem{ghosh_drift}
Arka Ghosh, Steven Noren, and Alexander Roitershtein, \emph{On the range
of the transient frog model on $\mathbb{Z}$}, Adv. in Appl. Probab.
\textbf{49} (2017), no.~2, 327--343.
\MR{3668379}

\bibitem{Fot?}
Jonathan Hermon, \emph{Frogs on trees?}, Electron. J. Probab.
\textbf{23} (2018), Paper No. 17, 40.
\MR{3771754}

\bibitem{HJJ2}
Christopher Hoffman, Tobias Johnson, and Matthew Junge, \emph{From
transience to recurrence with Poisson tree frogs},
Ann.\ Appl.\ Probab. \textbf{26} (2016), no.~3, 1620--1635.
\MR{3513600}

\bibitem{HJJ1}
Christopher Hoffman, Tobias Johnson, and Matthew Junge,
\emph{Recurrence and transience for the frog model on trees}, Ann.
Probab. \textbf{45} (2017), no.~5, 2826--2854.
\MR{3706732}

\bibitem{HJJ_cover}
Christopher Hoffman, Tobias Johnson, and Matthew Junge,
\emph{Cover time for the frog model on trees}, available at
\ARXIV{1802.03428}, 2018.
\MR{3706732}

\bibitem{JJ3_log}
Tobias Johnson and Matthew Junge, \emph{The critical density for the
frog model is the degree of the tree}, Electron.\ Commun.\ Probab.
\textbf{21} (2016), Paper No. 82, 12.
\MR{3580451}

\bibitem{JJ3_order}
Tobias Johnson and Matthew Junge, \emph{Stochastic orders and the frog model}, Ann. Inst. Henri
Poincar\'{e} Probab. Stat. \textbf{54} (2018), no.~2, 1013--1030.
\MR{3795075}

\bibitem{KS1}
Harry Kesten and Vladas Sidoravicius, \emph{The spread of a rumor or
infection in a moving population}, Ann. Probab. \textbf{33} (2005),
no.~6, 2402--2462.
\MR{2184100}

\bibitem{KS2}
Harry Kesten and Vladas Sidoravicius,
\emph{A phase transition in a model for the spread of an infection},
Illinois J. Math. \textbf{50} (2006), no.~1-4, 547--634.
\MR{2247840}

\bibitem{KS3}
Harry Kesten and Vladas Sidoravicius, \emph{A shape theorem for the spread of an infection}, Ann.
of Math. (2) \textbf{167} (2008), no.~3, 701--766.
\MR{2415386}

\bibitem{KZ}
Elena Kosygina and Martin P.~W. Zerner, \emph{A zero-one law for
recurrence and transience of frog processes}, Probab. Theory Related
Fields \textbf{168} (2017), no.~1-2, 317--346.
\MR{3651054}

\bibitem{MSH}
Neeraj Misra, Harshinder Singh, and E.~James Harner, \emph{Stochastic
comparisons of Poisson and binomial random variables with their
mixtures}, Statist. Probab. Lett. \textbf{65} (2003), no.~4, 279--290.
\MR{2039874 (2005a:60028)}

\bibitem{combustion}
Alejandro~F. Ram\'{\i }rez and Vladas Sidoravicius, \emph{Asymptotic
behavior of a stochastic combustion growth process}, J. Eur. Math. Soc.
(JEMS) \textbf{6} (2004), no.~3, 293--334.
\MR{2060478}

\bibitem{josh_drift2}
Josh Rosenberg, \emph{The nonhomogeneous frog model on $\mathbb{Z}$},
available at \ARXIV{1707.07749}, 2017.
\MR{3899930}

\bibitem{josh_32}
Josh Rosenberg,
\emph{Recurrence of the frog model on the 3,2-alternating tree},
available at \ARXIV{1701.02813}, 2017.
\MR{3840739}

\bibitem{josh_drift}
Joshua Rosenberg, \emph{The frog model with drift on $\mathbb{R}$},
Electron. Commun. Probab. \textbf{22} (2017), Paper No. 30, 14.
\MR{3663101}

\bibitem{SS}
Moshe Shaked and J.~George Shanthikumar, \emph{Stochastic orders},
Springer Series in Statistics, Springer, New York, 2007.
\MR{2265633}

\bibitem{TW}
Andr\'{a}s Telcs and Nicholas~C. Wormald, \emph{Branching and tree
indexed random walks on fractals}, J. Appl. Probab. \textbf{36} (1999),
no.~4, 999--1011.
\MR{1742145}

\end{thebibliography}




\end{document}